\documentclass[12pt]{amsart}
\usepackage{amssymb, amsmath, amsthm, latexsym, array}
\usepackage{fullpage,color}
\usepackage{verbatim,euscript}

\numberwithin{equation}{section}

\usepackage[colorlinks=true,linkcolor=blue,urlcolor=violet,citecolor=magenta]{hyperref}

\input {cyracc.def}
 \font\tencyr=wncyr10 
\font\tencyi=wncyi10 
\font\tencysc=wncysc10 
\def\rus{\tencyr\cyracc}
\def\rusi{\tencyi\cyracc}
\def\rusc{\tencysc\cyracc}

\newtheorem{thm}{Theorem}[section]

\newtheorem{lm}[thm]{Lemma}
\newtheorem{cl}[thm]{Corollary}
\newtheorem{prop}[thm]{Proposition}

\theoremstyle{remark}
\newtheorem{ex}[thm]{Example}
\newtheorem*{rem}{Remark}

\theoremstyle{definition}
\newtheorem{rmk}[thm]{Remark}
\newtheorem{df}{Definition}
\newtheorem*{qtn}{Question}

\newcommand {\ce}{{\mathfrak c}}

\newcommand {\g}{{\mathfrak g}}
\newcommand {\h}{{\mathfrak h}}

\newcommand {\q}{{\mathfrak q}}

\newcommand {\es}{{\mathfrak s}}
\newcommand {\te}{{\mathfrak t}}

\newcommand {\X}{{\mathfrak X}}
\newcommand {\Y}{{\mathfrak Y}}

\newcommand {\sltn}{{\mathfrak{sl}_{2n}}}

\newcommand {\N}{{\mathcal N}}
\newcommand {\co}{{\mathcal O}}
\newcommand {\cs}{{\mathcal S}}
\newcommand {\ml}{\mathcal}

\newcommand {\BZ}{{\mathbb Z}}
\newcommand {\BN}{{\mathbb N}}

\newcommand {\md}{/\!\!/}

\newcommand{\ap}{\alpha}
\newcommand{\vp}{\varphi}

\renewcommand{\le}{\leqslant}
\renewcommand{\ge}{\geqslant}

\newcommand{\eus}{\EuScript}

\newcommand {\ads}{{\mathrm{ad}^*}}

\newcommand {\codim}{{\mathrm{codim}}}

\newcommand {\gr}{{\mathrm{gr\,}}}

\newcommand {\hot}{\operatorname{\mathsf{ht}}}
\newcommand {\ind}{{\mathrm{ind\,}}}

\newcommand {\Ker}{{\mathrm{Ker\,}}}
\newcommand {\Ima}{{\mathrm{Im\,}}}
\newcommand {\Der}{\operatorname{\mathsf{Der}}}
\newcommand {\Mor}{\operatorname{\mathsf{Mor}}}
\newcommand {\Sym}{\operatorname{\mathsf{Sym}}}

\newcommand {\rk}{{\mathrm{rk\,}}}

\newcommand {\spe}{{\mathrm{Spec\,}}}

\newcommand {\tr}{{\mathrm{tr\,}}}
\newcommand {\trdeg}{{\mathrm{trdeg\,}}}

\newcommand {\tri}{\mathfrak{sl}_2}
\newcommand {\GR}[2]{{\textrm{{\bf #1}}}_{#2}}

\newcommand {\ov}{\overline}
\newcommand {\un}{\underline}
\newcommand {\bbk}{\Bbbk}
\begin{document}
\setlength{\parskip}{3pt plus 5pt minus 0pt}
\hfill { {\color{blue}\scriptsize August 15, 2008}}
\vskip1ex

\title[Adjoint vector fields and differential operators]
{Adjoint vector fields and differential operators on representation spaces}
\author[D.\,Panyushev]{Dmitri I.~Panyushev}
\address[]{Institute for Information Transmission Problems, B. Karetnyi per. 19, Moscow 101447\hfil\break\indent
Independent University of Moscow,
Bol'shoi Vlasevskii per. 11, 119002 Moscow, \ Russia}
\email{panyush@mccme.ru}
\urladdr{http://www.mccme.ru/~panyush}
\thanks{Supported in part by  R.F.B.R. grant
06--01--72550.}
\begin{abstract} 
Let $G$ be a semisimple algebraic group with Lie algebra $\g$.
In 1979, J.~Dixmier proved that any vector field annihilating all 
$G$-invariant polynomials on $\g$ lies in the $\bbk[\g]$-module generated by the 
"adjoint vector fields", i.e., vector fields $\varsigma$ of the form 
$\varsigma(y)(x)=[x,y]$, $x,y\in\g$.
A substantial generalisation of Dixmier's theorem was found
by Levasseur and Stafford. They explicitly described the centraliser of $\bbk[\g]^G$
in the algebra of differential operators on $\g$. On the level of vector fields,
their result reduces to Dixmier's theorem.
The purpose of this paper is to explore similar problems in the general context
of affine algebraic groups and their rational representations.
\end{abstract} 
\maketitle


\section*{Introduction}

\noindent
We work over an algebraically closed field $\bbk$ of characteristic zero.
Throughout, $G$ is a connected affine algebraic group with Lie algebra $\g$.
Suppose for a while that $G$ is semisimple.
In 1979, Jacques Dixmier proved a nice theorem on vector fields
on $\g$. Specifically, he showed that any vector field annihilating all 
$G$-invariant polynomials on $\g$ lies in the $\bbk[\g]$-module generated by the 
"adjoint vector fields" \cite[Theorem\,2.1]{dixm}. 
A substantial generalisation of Dixmier's theorem was found
by Levasseur and Stafford \cite{ls}.  They explicitly described the centraliser of $\bbk[\g]^G$
in the algebra of differential operators on $\g$. On the level of vector fields,
their result reduces to Dixmier's theorem.
The purpose of this paper is to explore similar problems in the general context
of affine algebraic groups and their rational representations.

We show that Dixmier's argument  applies to the coadjoint
representations of the so-called '3-wonderful' Lie algebras. Furthermore, the coadjoint representation can be replaced with an arbitrary (finite-dimensional) representation,
and this leads to  
three types of interesting problems. 
Let now $G$ be an arbitrary connected group. 
We say that $\g$ is {\it 3-wonderful\/} if: 
{\sf (i)} \ $\codim(\g^*\setminus \g^*_{reg})\ge 3$, 
{\sf (ii)} \ $\bbk[\g^*]^G$ is a polynomial algebra of Krull dimension $\ind\g$, and 
{\sf (iii)} \ the sum of degrees of free homogeneous generators of $\bbk[\g^*]^G$ 
equals $(\dim\g+\ind\g)/2$.  (Here $\g^*_{reg}$ is 
the union of $G$-orbits of maximal dimension in $\g^*$.)

This definition intends to axiomatise good properties of reductive Lie algebras.
There is also a method for generating new 3-wonderful algebras:  if $\g$ is $3$-wonderful and
$\g\simeq\g^*$, then the semi-direct product  $\g\ltimes\g$ has the same properties.
Below is a Dixmier-type result for the coadjoint representation of a 3-wonderful Lie algebra 
$\g$.

\begin{thm}   \label{thm:0.1}
Let $\X$ be a polynomial vector field on $\g^*$. Assume that $\X$ annihilates
all of\/ $\bbk[\g^*]^G$. Then there is a polynomial mapping $\Y:\g^*\to \g$
such that $\X(\xi)=\Y(\xi){\cdot}\xi$ for any $\xi\in\g^*$.
\end{thm}

The proof is essentially based on the fact that $\codim(\g^*\setminus \g^*_{reg})\ge 3$ 
and certain
vector bundle on $\g^*_{reg}$ appears to be trivial. A posteriori, this is related to 
some good properties of a homomorphism of $\bbk[\g^*]$-modules.
Let $\Mor(V,N)$ denote the set of all polynomial morphisms $V\to N$, where $V$ and $N$
are $\bbk$-vector spaces. It is a free graded $\bbk[V]$-module of rank $\dim N$.
Consider the  homomorphism 
$\hat\phi :\Mor(\g^*,\g)\to \Mor(\g^*,\g^*)$,
$\hat\phi(F)(\xi):= F(\xi){\cdot}\xi$ for $\xi\in\g^*$. 
Then Theorem~\ref{thm:0.1} merely says that 
$\Ima\hat\phi$ equals the submodule of vector fields annihilating all of $\bbk[\g^*]^G$.
Furthermore, in the "3-wonderful case" the kernel of $\hat\phi$ appears to be a free 
$\bbk[\g^*]$-module generated by $G$-equivariant morphisms.

For an arbitrary $G$-module $V$, where $\g$ is not necessarily $3$-wonderful, 
one can write up {\sl three\/} similar homomorphisms and consider similar problems. 
The most obvious possibility is to replace 
$\g^*$ with $V$.  This yields the homomorphism of $\bbk[V]$-modules
$\hat\phi:\Mor(V,\g)\to \Mor(V,V)$. Clearly,  any vector field $\X\in\Ima\hat\phi$ annihilates all 
of $\bbk[V]^G$. The problem on the opposite inclusion is related to 
the structure of $\Ker\hat\phi$, $\bbk[V]^G$, and $V_{reg}$, and we provide an appropriate 
analogue of Theorem~\ref{thm:0.1}.
Two other possibilities are
\[
\hat\psi: \Mor(V,V^*){\to} \Mor(V,\g^*) \ \text{ and } \
\hat\tau: \Mor(\g,V) {\to} \Mor(\g,V) ,
\]
where we also describe the respective images under similar conditions, and give some
illustrations.

Generalising the approach of \cite{ls}, we regard the problem on  
$\Ima\hat\phi$ as a special case of a problem on differential operators on $V$.
Let $\ml D(V)$ denote the ring of differential operators on $V$ with polynomial coefficients.
As $\bbk[V]$ is identified with the differential operators of order zero,  one can consider
the centraliser of $\bbk[V]^G$ in $\ml D(V)$,
$\mathsf{Cent}_{\ml D(V)}(\bbk[V]^G)$. Each $x\in\g$ gives rise to a linear operator on $V$
and therefore a vector field. In this way, one obtains the Lie algebra homomorphism 
\[
\boldsymbol{\varsigma}=\boldsymbol{\varsigma}_V: \g\to \{\text{polynomial vector fields on $V$}\}
\subset \ml D(V).
\] 
The elements of $\boldsymbol{\varsigma}(\g)$ are called the {\it adjoint vector fields} (on $V$). By the definition of $\hat\phi$, 
$\Ima\hat\phi$ is the $\bbk[V]$-module generated by $\boldsymbol{\varsigma}(\g)$.
Clearly, $\eus C:=\mathsf{Cent}_{\ml D(V)}(\bbk[V]^G)$ contains $\bbk[V]$ and $\boldsymbol{\varsigma}(\g)$, 
and one may ask whether the subalgebra generated by $\bbk[V]$ and $\boldsymbol{\varsigma}(\g)$, denoted 
$\eus A$, is equal to $\eus C$. On the level of vector fields, the equality $\eus A=\eus C$
reduces to the assertion
that any $\X$ annihilating all of $\bbk[V]^G$ lies in $\Ima\hat\phi$, i.e., in
$\bbk[V]\boldsymbol{\varsigma}(\g)$.

For the adjoint representation of a semisimple Lie algebra, the equality $\eus A=\eus C$ 
is proved in \cite{ls}. Adapting that method, we obtain a sufficient condition for 
$\eus A=\eus C$ in a more general framework. 
We assume that
$\bbk[V]^G$ is a polynomial algebra and $\Ker\hat\phi$ is a free $\bbk[V]$-module, and
impose determinantal constraints on the embedding $\Ker\hat\phi\to \Mor(V,\g)$.
However, the reductivity of $G$ is not assumed. 
(See 
Theorem~\ref{thm:main4-2} for precise formulations). 
The equality $\eus A=\eus C$  and other results on $\eus C$ stem from 
assertions about  certain $G$-stable subvariety of $V\times \g^*$.
Recall that for any $G$-module $V$, there is the moment map 
$\mu: V\times V^*\to \g^*$. Then $\varkappa:
V\times V^* \to V\times \g^*$ is defined
by letting $\varkappa(v,\xi)=(v, \mu(v,\xi))$. Consider also the $\bbk[V]$-module
$E=\Ima\hat\phi$ and its symmetric algebra,  $\mathsf{Sym}_{\bbk[V]}(E)$.
Under appropriate constraints (alluded to above), we prove that
\begin{itemize}
\item  $\mathsf{Sym}_{\bbk[V]}(E)$ is a factorial domain of Krull dimension $\dim V+\dim\g-
\rk(\Ker\hat\phi)$;
\item   
$\ov{\Ima\varkappa}=\spe(\mathsf{Sym}_{\bbk[V]}(E))$ and it is also a complete intersection in 
$V\times \g^*$;
\item the generators of the ideal of $\ov{\Ima\varkappa}$
 are determined by a basis of $\Ker\hat\phi$. 
\end{itemize}
From this, we deduce that
$\gr\eus A=\gr\eus C=\mathsf{Sym}_{\bbk[V]}(E)$, where $\gr(.)$ is the associated graded
ring with respect to the  filtration by the order of differential operators.
Then the equality $\eus A=\eus C$ follows. We also give a sufficient condition for
$\eus C$ to be a free $\bbk[V]^G$-module (see Theorem~\ref{thm:C_is_free}). 

Let  $\g=\g_0\oplus\g_1$ be a  $\BZ_2$-graded semisimple Lie algebra.
This grading (or the symmetric pair $(\g,\g_0)$) is said to be $\N$-{\it regular\/}, if
$\g_1$ contains a regular nilpotent element of $\g$.
Our main application concerns the isotropy representation $(G_0{:}\g_1)$.
We show that the hypotheses of Theorems~\ref{thm:main4-2} and \ref{thm:C_is_free}
are satisfied, modulo one exception, for $(G_0{:}\g_1)$  if  $(\g,\g_0)$ is 
$\N$-regular. Hence  $\eus A=\eus C$ and 
$\eus C$ is a free $\bbk[\g_1]^{G_0}$-module  
in these cases. Verification of all necessary conditions requires a detailed information
on the structure of the null-cone in the $G_0$-module $\g_1$.
We also provide other examples of the ``$\eus A=\eus C$ phenomenon'';
in particular, those for the coadjoint representation of non-reductive Lie algebras.

The plan of the article is as follows.
Section~\ref{prelim} contains preliminaries on group actions and differential operators.
In Section~\ref{sect:kost}, we prove our analogue of Dixmier's result for the coadjoint representation of a $3$-wonderful Lie algebra. 
Then we discuss, in Section~\ref{sect:mod}, 
three generalisations to the case in which
$\g^*$ is replaced with an arbitrary $G$-module. Section~\ref{sect:diff_op} 
contains our results on the image of $\varkappa$, $\mathsf{Sym}_{\bbk[V]}(E)$, 
and the equality $\eus A=\eus C$. In Section~\ref{sect:z2}, 
we consider applications to $\BZ_2$-gradings of semisimple Lie algebras and provide some
other examples.
In Section~\ref{sect:some_spec}, we discuss possible connections between our 
results for $\BZ_2$-gradings and another generalisation of Dixmier's result
obtained in \cite{ls96, ls99}.

\un{\sl Some notation.}
If an algebraic group $G$ acts on an irreducible affine variety $X$, then $\bbk[X]^G$ 
is the algebra of $G$-invariant regular functions on $X$ and $\bbk(X)^G$
is the field of $G$-invariant rational functions. If $\bbk[X]^G$
is finitely generated, then $X\md G:=\spe \bbk[X]^G$, and
the {\it quotient morphism\/} $\pi_X: X\to X\md G$ is the mapping associated with
the embedding $\bbk[X]^G \hookrightarrow \bbk[X]$.
We use dot `$\cdot$' to denote the action of (elements of) $G$ and $\g$ on $X$.
For instance, $G{\cdot}x$ is the orbit of $x\in X$. 
The stabiliser of $x$ in $\g$ is denoted  by $\g_x$.
\\   \indent
All topological terms refer to the Zariski topology.
The pairing of dual vector spaces is denoted by $\langle\ ,\ \rangle$.
If $M$ is a subset of a vector space, then $\mathsf{span}(M)$ denotes the linear span of
$M$. 

{\small {\bf Acknowledgements.}  
I am grateful to Thierry Levasseur for sending me
his unpublished notes and useful e-mail exchange.}

\section{Preliminaries} 
\label{prelim}

\subsection{}
Let $G$ be an affine algebraic group acting regularly on an 
irreducible algebraic variety $X$. 
We say that $\h\subset \g$ is {\it a generic stabiliser\/} for the 
the action $(G:X)$
if there exists
a dense open subset $\Omega\subset X$ such that all stabilisers $\g_x$, $x\in \Omega$,
are $G$-conjugate to $\h$. 
The points of such an $\Omega$ are said to be {\it generic}.
Generic stabilisers always exist if $G$ is reductive and  $X$ is  smooth \cite{Ri}.

Let $X_{reg}$ denote the set of all {\it regular\/} elements of $X$. That is,
\begin{multline*}
   X_{reg}:=
   \{ x\in X\mid \dim G{\cdot}x \ge \dim G{\cdot}x' \text{ for all } x' \in X\}= \\
    \{ x\in X\mid \dim \g_x \le \dim \g_{x'} \text{ for all } x' \in X\} \ .
\end{multline*}
As is well-known, $X_{reg}$ is a dense open subset of $X$. If we want to explicitly specify 
the group acting on $X$, we refer to $G$-{\it regular\/} elements.
\begin{df}   \label{def:codim}
A $G$-variety $X$ is said to have the {\it codim--$n$ property\/} if
$\codim_X (X\setminus X_{reg})\ge n$.
\end{df}

\noindent
We will mostly use this notion if $X=V$ is a $G$-module.

\noindent
{\bf Example.} Let  $\g$ be reductive and $\N\subset\g$ the nilpotent cone.
Then $\g$ (resp. $\N$) has the codim--$3$ (resp. codim--$2$) property
with respect to the adjoint representation \cite{ko63}.

Recall that the {\it index\/} of $\g$, denoted $\ind\g$, 
is the minimal dimension of stabilisers for the elements
of the $\g$-module $\g^*$.
That is, $\ind\g=\displaystyle \min_{\xi\in\g^*} \dim \g_\xi= \dim \g_\eta$
for any $\eta\in \g^*_{reg}$.

\subsection{}
For finite-dimensional $\bbk$-vector spaces $V$ and $N$, let 
$\Mor(V,N)$ denote the set of 
{\sl polynomial\/} morphisms $V\to N$. Clearly, $\Mor(V,N)\simeq \bbk[V]\otimes N$
and it is a free graded $\bbk[V]$-module of rank $\dim N$.

If $V$ and $N$ are $G$-modules, then $G$ acts on $\Mor(V,N)$ by the rule
$(g\ast F)(v)=g{\cdot}(F(g^{-1}{\cdot}v))$. Then  $(\Mor(V,N))^G=:\Mor_G(V,N)$
is the set of all polynomial
$G$-equivariant morphisms $V\to N$. It is a $\bbk[V]^G$-module, which is called
{\it the module of covariants of type $N$}.
If $G$ is reductive, then the algebra $\bbk[V]^G$ is
finitely generated and each $\Mor_G(V,N)$
is a finitely generated $\bbk[V]^G$-module.

[All these constructions makes sense if $V$ is replaced with any affine $G$-variety $X$.]

\subsection{}
Let $\ml D(V)$ denote the algebra of differential operators on $V$, with polynomial coefficients.
Recall that $\ml D(V)$ contains the symmetric algebra of $V$,\/ $\cs(V)$, as the subalgebra of constant coefficient differential operators and $\bbk[V]$ as the subalgebra of differential operators of order zero.
We always filter $\ml D(V)$ by the order of differential operators, hence
$\mathrm{gr}_n\ml D(V)\simeq \bbk[V]\otimes \cs^n(V)$ and 
$\gr\ml D(V)$ is isomorphic to $\bbk[V]\otimes \cs(V)=\bbk[V\times V^*]$ as algebras.

Let $\mathsf{Der}(\bbk[V])$ denotes the $\bbk[V]$-module of all $\bbk$-derivations 
of $\bbk[V]$ or, equivalently, the module of polynomial vector fields on $V$. 
Then $\mathsf{Der}(\bbk[V])\simeq \Mor(V,V)$. 
A vector field $\X$ can be regarded either as polynomial endomorphism of
$V$  or as  linear endomorphism of $\bbk[V]$. The respective notation is $\X(v)$, $v\in V$
and $\X\{f\}$, $f\in\bbk[V]$.

\section{Adjoint vector fields and 3-wonderful Lie algebras}
\label{sect:kost}

\noindent
In this section, $G$ is a connected algebraic group.

Let $V$ be a (finite-dimensional, rational) $G$-module.
The differential of the $G$-action on $\bbk[V]$ yields a map 
$\boldsymbol{\varsigma}=\boldsymbol{\varsigma}_V:\g\to \mathsf{Der}(\bbk[V]) \subset \ml D(V)$. 
Upon the identification $\mathsf{Der}(\bbk[V])$ with $\Mor(V,V)$, we see 
that $\boldsymbol{\varsigma}(e)$ is just the linear operator on $V$ corresponding to $e\in\g$.
The vector fields on $V$ of the form $\boldsymbol{\varsigma}(e)$ are said to be  the {\it adjoint vector fields}.
For  $\g$ semisimple and $V=\g$,
Dixmier describes a relationship between the adjoint vector fields and vector
fields annihilating all of $\bbk[\g]^G$ \cite[Theorem\,2.1]{dixm}. 
Below, we prove that this result naturally extends to the coadjoint representations of certain 
non-reductive Lie algebras.

\noindent
In \cite{ko63}, Kostant  established a number of fundamental properties of complex reductive Lie algebras. Motivated by these results, we give the following 

\begin{df}    \label{def:kost}
An algebraic Lie algebra $\g$ is said to be $n$-{\it wonderful}, if the following conditions are satisfied:
\begin{itemize}
\item[\sf (i)] \ the coadjoint representation of $\g$ has the codim--$n$ property.
\item[\sf (ii)] \  $\bbk[\g^*]^G$ is a polynomial algebra of Krull dimension $l=\ind\g$;
\item[\sf (iii)] \ If $f_1,\ldots,f_l$ are homogeneous algebraically independent generators
of $\bbk[\g^*]^G$, then $\sum_{i=1}^l \deg f_i=(\dim\g+\ind\g)/2$;
\end{itemize}
\end{df}

\begin{rmk}   \label{rmk:3wond-conn}
We are only interested in $n$-wonderful algebras for $n=2,3$.
Let us point out  some connections between hypotheses of this definition, and  their  consequences.

1. For any Lie algebra, $\trdeg\bbk(\g^*)^G$ equals 
$\ind\g$ and hence $\bbk[\g^*]^G$ contains at most $\ind\g$ algebraically independent
elements.
Thus, condition (ii) also means that $(\g,\ads)$ has sufficiently many 
polynomial invariants.

2. If $(\g,\ads)$ has the codim--$2$ property and 
$f_1,\ldots,f_l\in \bbk[\g^*]^G$ are algebraically independent, then
$\sum_{i=1}^l \deg f_i\ge (\dim\g+\ind\g)/2$. Furthermore, if
the equality holds, then  $f_1,\ldots,f_l$ freely generate
$\bbk[\g^*]^G$  and 
\begin{equation}   \label{ravno}
    \g^*_{reg}=\{\xi\in\g^* \mid (\textsl{d}f_1)_\xi, \ldots,(\textsl{d}f_l)_\xi 
  \ \text{ are linearly independent}  \},
\end{equation}
see \cite[Theorem\,1.2]{coadj}. It follows that Eq.~\eqref{ravno} holds for any
$2$-wonderful algebra. 
For $\g$ reductive,  equality \eqref{ravno} is a celebrated result of Kostant
\cite[Theorem\,0.1]{ko63}.

3. The main result of \cite{codim3} asserts that if $\g$ is $3$-wonderful, then 
the Poisson commutative subalgebra of $\bbk[\g^*]$ obtained from $\bbk[\g^*]^G$
via the argument shift method is {\it maximal} for any $\xi\in\g^*_{reg}$.

4. Any reductive Lie algebra is $3$-wonderful.
Several non-trivial examples of $3$-wonderful algebras are discussed in 
\cite[Section\,4]{codim3}.
\end{rmk}

\begin{thm}   \label{ala}
Let $\g$ be a $3$-wonderful Lie algebra. Given a polynomial vector field $\X$ on $\g^*$, the
following conditions are equivalent:
\begin{itemize}
\item[\sf (i)] \  $\X$ annihilates all $G$-invariant polynomials on $\g^*$;
\item[\sf (ii)] \ $\X(\xi)\in \g{\cdot}\xi$ for any $\xi\in \g^*_{reg}$;
\item[\sf (iii)] \ $\X(\xi)\in \g{\cdot}\xi$ for any $\xi\in \g^*$;
\item[\sf (iv)] \ There is a polynomial mapping $\Y\in\Mor (\g^*,\g)$ such that
$\X(\xi)=\Y(\xi){\cdot}\xi$ for any $\xi\in \g^*$.
\end{itemize}
\end{thm}\begin{proof}
Recall that for $f\in \bbk[\g^*]$ the polynomial $\X\{f\}\in \bbk[\g^*]$ is defined by
\begin{equation}   \label{vf}
   \X\{f\}(\xi)=\langle \X(\xi), (\textsl{d}f)_\xi \rangle .
\end{equation}
It is therefore clear that (iv)$\Rightarrow$(iii)$\Rightarrow$(ii)$\Rightarrow$(i).
It remains to prove the implication (i)$\Rightarrow$(iv). 
To this end, we need some preparations. Up to some obvious alterations, the rest of the 
proof is a repetition of the proof of Theorem~2.1 in \cite{dixm}.

Set $\Omega=\g^*_{reg}$. If $\xi\in\Omega$, then
$(\textsl{d}f_1)_\xi, \ldots,(\textsl{d}f_l)_\xi$ form a  basis for $\g_\xi$, in view of
Definition~\ref{def:kost} and Eq.~\eqref{ravno}.

Let $E$ be the cotangent bundle of $\Omega$, which is identified with  
$E\simeq \Omega\times \g$.
Let $E'$ be the sub-bundle of $E$ whose fibre of $\xi$ is $\g_\xi$.
The previous paragraph shows that the $\textsl{d}f_i$'s yield a trivialisation of 
$E'$.
Let $E''$ be the sub-bundle of the tangent bundle of $\Omega$ whose fibre of
$\xi$ is $\g{\cdot}\xi$. Since the kernel of the surjective mapping
$(x\in \g) \mapsto (x{\cdot}\xi\in \g{\cdot}\xi)$ is $\g_\xi$, 
one obtains the exact sequence of vector bundles
\[
    0\to E'\to E\to E''\to 0 \ 
\] 
and the exact sequence 
\[
   H^0(\Omega, E)\to H^0(\Omega, E'')\to H^1(\Omega, E') \ .
\]
Let $\eus O$ denote the structure sheaf of $\g^*$. By \cite[cor.\,2.9, p.16]{SGA}, 
there exists an exact
sequence of cohomology groups
\[
   H^1(\g^*,\eus O)\to H^1(\Omega,\eus O\vert_\Omega)\to H^2_{\g^*\setminus\Omega}(\g^*,
   \eus O) \ .
\]
Here $H^1(\g^*,\eus O)=0$ because $\g^*$ is affine, and it follows from the codim--$3$ 
property that $H^2_{\g^*\setminus\Omega}(\g^*, \eus O)=0$ \cite[cor.\,1.4, p.80]{SGA}.
Hence $H^1(\Omega,\eus O\vert_\Omega)=0$. This fact and the trivilality of $E'$
imply that $H^1(\Omega, E')=0$. Thus, the homomorphism 
$\gamma: H^0(\Omega, E)\to H^0(\Omega, E'')$ is onto.

Suppose that $\X$ satisfies assumption (i). Then Eq.~\eqref{vf}
and the linear independence of the differentials $(\textsl{d}f_i)_\xi$, $\xi\in\Omega$,  
show that $\X$ also satisfies (ii).
Therefore $\X\vert_\Omega$ is a section of $E''$.
The  surjectivity of $\gamma$ means that there exists a polynomial mapping
$\Y_0:\Omega\to \g$ such that $\X(\xi)=\Y_0(\xi){\cdot}\xi$ for any
$\xi\in\Omega$. Since $\codim(\g^*\setminus\Omega)\ge 2$, 
$\Y_0$ extends to a polynomial mapping $\Y:\Omega\to \g$, and 
the equality $\X(\xi)=\Y(\xi){\cdot}\xi$ holds  for all $\xi\in\g^*$.
\end{proof}

\begin{rmk}
This theorem is a statement about the coadjoint representation of $G$. There are two key
points in the proof. First, $E'$ appears to be a trivial bundle and, second, 
$\codim (\g^*\setminus \Omega)\ge 3$. Using this observation, we show in 
Section~\ref{sect:mod} that Theorem~\ref{ala} admits various generalisations to
other representations of $G$.
\end{rmk}

\begin{rmk}    \label{istolkov}
We know that $\mathsf{Der}(\bbk[\g^*])$ is a $\bbk[\g^*]$-module
and  $\boldsymbol{\varsigma}(\g)\subset \mathsf{Der}(\bbk[\g^*])$, where $\boldsymbol{\varsigma}=\boldsymbol{\varsigma}_{\g^*}$.
Therefore,  implication (i)$\Rightarrow $(iv) in Theorem~\ref{ala}
can be stated as follows:

{\sl If\/ $\X\in \mathsf{Der}(\bbk[\g^*])$ annihilates all of\/ $\bbk[\g^*]^G$, then 
$\X\in \bbk[\g^*]{\cdot}\boldsymbol{\varsigma}(\g)$.}
\end{rmk}

\begin{rmk}    \label{essence}
For an arbitrary Lie algebra $\g$, define
the homomorphism of $\bbk[\g^*]$-modules
$\hat\phi: \Mor(\g^*,\g)\to \Mor(\g^*,\g^*)$  by
$\hat\phi(F)(\xi)=F(\xi){\cdot}\xi$, $\xi\in\g^*$. Then 
\[
\Ima\hat\phi\subset \{\eus F: \g^*\to \g^* \mid  \eus F(\xi)\in \g{\cdot}\xi \quad \forall\xi\in\g^* \}
=:\eus T \ .
\]
Since the elements of $\Mor(\g^*,\g^*)$ is are just the vector field on $\g^*$,
the equivalence of conditions (iii) and (iv) in Theorem~\ref{ala} reduces to
the assertion that if $\g$ is $3$-wonderful, then $\Ima\hat\phi=\eus T$.  

Notice that $\Ker\hat\phi=\{F:\g^*\to \g\mid F(\xi)\in \g_\xi \quad \forall\xi\in\g^* \}$.
If $\Ker\hat\phi$ is a {\it free\/} $\bbk[\g^*]$-module (of rank $l=\ind\g$) and
$F_1,\ldots,F_l$ is a basis, then 
$E'$ is a trivial vector bundle over $\Omega'=\{\xi\in \g^*_{reg}\mid F_1(\xi),\dots, F_l(\xi)
\ \text{ are linearly independent}\}$. 
For any $f\in\bbk[\g^*]^G$, we have $\textsl{d}f\in \Ker\hat\phi\cap \Mor_G(\g^*,\g)$.
If $\g$ is $2$-wonderful, then  \cite[Theorem\,1.9]{jac} applies to the coadjoint
representation of $\g$, and one concludes that $\Ker\hat\phi$ is freely generated
by  the differentials $\textsl{d}f_1,\ldots,\textsl{d}f_l$.
Then, using Eq.~\eqref{ravno}, we obtain $\Omega'=\g^*_{reg}$.
This argument shows that in some cases (actually, most interesting ones), the triviality of $E'$
is closely related to the fact that $\Ker\hat\phi$ is a free $\bbk[\g^*]$-module 
generated by $G$-equivariant morphisms. 
\end{rmk}

\begin{ex}     \label{ex:new_wonder}
There is a procedure that generates new $n$-wonderful algebras from old ones (for $n\ge 2$).
Let $\q$ be a quadratic $n$-wonderful Lie algebra
("quadratic" means that $\q^*\simeq \q$ as $\q$-module) . 
Form the semi-direct product $\g=\q\ltimes\q$
(the second copy of $\q$ is an Abelian ideal of $\g$).  
Then $\g$ is again quadratic and $n$-wonderful. 
That $\g$ to be quadratic is elementary.  Therefore we can deal with the adjoint
representation of $\g$.
It then suffices to apply Theorem~7.1 in \cite{p05} to the case $V=\q$.
Roughly speaking, that theorem says that
the passage $\q \mapsto \q\ltimes\q$ doubles all  data
occurring in Definition~\ref{def:kost}. That is, $\dim\g=2\dim\q$, $\ind\g=2\,\ind\q$;
each basis invariant $f_i\in \bbk[\q]^Q$ gives rise to two basis invariants
in $\bbk[\g]^G$, and the degree for all three is the same. 
Finally, it is easily seen that $\q_{reg}\ltimes\q\subset \g_{reg}$.
Hence the codim--$n$ property is also preserved.

In particular, one can start with any semisimple $\es$ and take
$\es\ltimes\es$. This yields interesting examples of $3$-wonderful algebras. 
Notice that then this procedure can be iterated ad infinum. 
\end{ex}

\section{Modules over polynomial rings  associated with representations}
\label{sect:mod}

\noindent Unless otherwise stated, $G$ is an arbitrary connected algebraic group.
Let $V$ be a $G$-module. 
Associated with $V$, $\g$, and $\g^*$, there are at least three natural exact sequences of
modules over polynomial rings:

\begin{itemize}
\item[(A)] \qquad  $0\to\Ker\hat\phi\to \Mor(V,\g)\stackrel{\hat\phi}{\to} \Mor(V,V)$,
\item[(B)] \qquad  $0\to\Ker\hat\psi\to \Mor(V,V^*)\stackrel{\hat\psi}{\to} \Mor(V,\g^*)$,
\item[(C)] \qquad  $0\to\Ker\hat\tau\to \Mor(\g,V)\stackrel{\hat\tau}{\to} \Mor(\g,V)$.
\end{itemize}

\noindent
The first two sequences consist of $\bbk[V]$-modules, and the last one consists of 
$\bbk[\g]$-modules. Some of the properties of (A) and (B) 
have been studied in \cite{jac}, whereas  (B) and (C) have also been considered in
\cite[Sect.\,8]{p05}. 
Recall the definitions of $\hat\phi,\hat\psi,{\hat\tau}$:

(A) \quad  ${\hat\phi}(F)(v):=F(v){\cdot}v$, where $v\in V$;

(B) \quad $\langle \hat\psi(F)(v), x\rangle:=\langle x{\cdot}v, F(v)\rangle$, where $v\in V$,
$x\in\g$, 
and $\langle\ ,\ \rangle$ stands for  the pairing of elements of dual vector spaces. 
One can also exploit the {\it moment mapping\/} $\mu: V\times V^* \to \g^*$, 
which is defined by
$\langle \mu(v,\eta), x\rangle=\langle x{\cdot}v,\eta\rangle$, where $\eta\in V^*$.
Then
$\hat\psi(F)(v):=\mu(v,F(v))$. 

(C) \quad ${\hat\tau}(F)(x):=x{\cdot}F(x)$, where $x\in \g$.

\begin{rem}
$\Ker\hat\phi$ is a $G$-stable submodule  of $\Mor(V,\g)$; and likewise for 
$\Ker\hat\psi$ and $\Ker\hat\tau$.
\end{rem}
\noindent
Note that, for $V=\g^*$, the sequences (A) and (B) coincide,  and we obtain the situation 
of Remark~\ref{essence}. Also, the sequences (A) and (C) coincide if $V=\g$.
Below we formulate Dixmier-type statements, which characterise the images of
$\hat\phi,\hat\psi$, and ${\hat\tau}$ under similar (rather restrictive) assumptions.

\noindent {\bf Case (A)}. 
Here  $\Ker\hat\phi=\{F:V\to \g\mid F(v)\in\g_v\ \ \forall v\in V \}$ and
$\Ima\hat\phi\subset \{\eus F:V\to V\mid \eus F(v)\in \g{\cdot}v \ \ \forall v\in V \}$.
Set $\Omega_\phi=V_{reg}$. Consider three vector bundles on $\Omega_\phi$:
\begin{gather*}
E'_\phi=\{(v,x)\mid x{\cdot}v=0\}=\{(v,x)) \mid x\in \g_v\}
\subset \Omega_\phi\times \g,  \\
E_\phi= \Omega_\phi\times \g, \quad
E''_\phi=\{(v,x{\cdot}v)\mid v\in\Omega_\phi,\ x\in\g\} \subset \Omega_\phi\times V
\end{gather*}
and the corresponding exact sequence $0\to E'_\phi\to E_\phi\to E''_\phi\to 0$. 
Arguing as in the proof of Theorem~\ref{ala}, one obtains

\begin{prop}   \label{prop-a}
Suppose $E'_\phi$ is a trivial vector bundle and $\codim(V\setminus\Omega_\phi)\ge 3$. Then
\[
\Ima\hat\phi= \{\eus F:V\to V\mid \eus F(v)\in \g{\cdot}v\ \ \forall v\in V \} .
\]
In other words, if $\eus F(v)\in \g{\cdot}v$ for all $v\in\Omega_\phi$, then there is
$F:V\to\g$ such that $\eus F(v)=F(v){\cdot}v$ for all $v\in V$.
\end{prop}

\noindent 
This is not a complete analogue of Theorem~\ref{ala}, since we obtain only equivalence of
the following three conditions on the vector field $\eus F:V\to V$:

\begin{itemize}
\item[\sf (ii)] \ $\eus F(v)\in \g{\cdot}v$ for any $v\in V_{reg}$;

\item[\sf  (iii)] \ $\eus F(v)\in \g{\cdot}v$ for any $v\in V$;
 
\item[\sf (iv)]  There is an $F\in \Mor(V,\g)$ such that
$\eus F(v)=F(v){\cdot}\xi$ for any $\xi\in V$.
\end{itemize}
In order to add condition 
\begin{itemize}
\item[\sf (i)] \  $\eus F$ annihilates all of $\bbk[V]^G$
\end{itemize}
to  this list, one has to impose some constraints on $\bbk[V]^G$. For instance, it suffices to require that the quotient field of $\bbk[V]^G$ equals to $\bbk(V)^G$ and that
$\dim (\mathsf{span}\{\textsl{d}f_v \mid f\in \bbk[V]^G\})=\trdeg \bbk(V)^G$ for any 
$v\in V_{reg}$. (Cf. the proof of Theorem~\ref{ala}).
Actually, these two conditions are not too restrictive. These are always satisfied if
$G$ is semisimple and $\bbk[V]^G$ is a polynomial (free) algebra (see \cite{knop}).

\noindent The problem  of triviality for $E'_\phi$ is connected with the
question of whether $\Ker\hat\phi$ is a free $\bbk[V]$-module. This seems to be related 
to the property that a generic stabiliser for $(\g:V)$ is abelian. 
In the following sections, we study case (A) more carefully, prove a more general result,
and provide some examples.

\noindent {\bf Case (B)}.
Here  $\Ker\hat\psi=\{F:V\to V^*\mid \mu(v, F(v))=0\ \ \forall v\in V \}$ and
$\Ima\hat\psi\subset \{\eus F:V\to \g^*\mid \eus F(v)\in \mu(v,V^*)\ \ \forall v\in V \}$.
Again, we take $\Omega_\psi=V_{reg}$. Consider three vector bundles on $\Omega_\psi$:
\begin{gather*}
E'_\psi=\{(v,\xi)\mid \mu(v,\xi)=0\}=\{(v,\xi)\mid \xi\in (\g{\cdot}v)^\perp\} 
\subset\Omega_\psi\times V^*,  \\
E_\psi= \Omega_\psi\times V^*, \quad
E''_\psi=\{(v,\mu(v,\xi))\mid v\in\Omega_\psi,\ \xi\in V^*\} \subset \Omega_\psi\times \g^*
\end{gather*}
and the corresponding exact sequence $0\to E'_\psi\to E_\psi\to E''_\psi\to 0$. 
Arguing as in the proof
of Theorem~\ref{ala}, one obtains

\begin{prop}   \label{prop-b}
Suppose $E'_\psi$ is a trivial vector bundle and $\codim(V\setminus\Omega_\psi)\ge 3$. Then
\[
\Ima\hat\psi= \{\eus F:V\to \g^*\mid \eus F(v)\in \mu(v,V^*)\ \ \forall v\in V \} .
\]
In other words, if $\eus F(v)\in \mu(v,V^*)$ for all $v\in\Omega_\psi$, then there is
$F:V\to V^*$ such that $\eus F(v)=\mu(v, F(v))$ for all $v\in V$.
\end{prop}

The hypotheses of Proposition~\ref{prop-b} are  satisfied in the following situation.

\begin{thm} \label{thm-b}
Suppose $\g$ is semisimple, $\bbk[V]^G$ is polynomial, and $\codim (V\setminus V_{reg})
\ge 3$. Then $\Ker\hat\psi$ is a free $\bbk[V]$-module and 
Proposition~\ref{prop-b} applies.
\end{thm}\begin{proof} 
Let $f_1,\ldots,f_l\in \bbk[V]^G$ be the basis invariants.
By \cite[1.9\,{\&}\,1.10]{jac}, $\Ker\hat\psi$ is a free $\bbk[V]$-module
generated by $\textsl{d}f_1, \ldots,\textsl{d}f_l$.
By \cite[Korollar\,1]{knop},  $V_{reg}\subset \{v\in V\mid 
(\textsl{d}f_1)_v, \ldots,(\textsl{d}f_l)_v \ \text{ are linearly independent}\}$.
It follows that $E'_\psi$ is a trivial bundle on $V_{reg}$.
\end{proof}

\noindent  It follows from \cite[Remark\,4.5]{codim3} that, under the assumptions of 
Theorem~\ref{thm-b}, a generic stabiliser for $(G:V)$ has to be non-trivial, i.e., 
$\max\dim G{\cdot}v < \dim V$.

\noindent {\bf Case (C)}.
For $x\in \g$, we set $V^x=\{v\in V\mid x{\cdot}v=v\}$.
Here  $\Ker\hat\tau=\{F:\g\to V\mid F(x)\in V^x\ \ \forall x\in \g \}$ and
$\Ima\hat\tau\subset \{\eus F:\g\to V\mid \eus F(x)\in x{\cdot}V\ \ \forall x\in \g \}$.
Set $\Omega_\tau=\{x\in \g\mid  \dim V^x \text{ is minimal}\}$.
Consider three vector bundles on $\Omega_\tau$:
\begin{gather*}
E'_\tau=\{(x,v)\mid x{\cdot}v=0\}=\{(x,v) \mid v\in V^x\}
 \subset\Omega_\tau\times V,  \quad
E_\tau= \Omega_\tau\times V, \\
E''_\tau=\{(x,x{\cdot}v )\mid x\in\Omega_\tau,\ v\in V\} \subset \Omega_\tau\times V
\end{gather*}
and the corresponding exact sequence $0\to E'_\tau\to E_\tau\to E''_\tau\to 0$. 
Arguing as in the proof
of Theorem~\ref{ala}, one obtains

\begin{prop}   \label{prop-c}
Suppose $E'_\tau$ is a trivial vector bundle and $\codim(\g\setminus\Omega_\tau)\ge 3$. Then
\[
\Ima\hat\tau= \{\eus F:\g\to V\mid \eus F(v)\in x{\cdot}V\ \ \forall x\in \g \} .
\]
In other words, if $\eus F(x)\in x{\cdot}V$ for all $x\in\Omega_\tau$, then there is
$F:\g\to V$ such that $\eus F(x)=x{\cdot}F(x)$ for all $x\in q$.
\end{prop}

\noindent
It is remarkable that if $G$ is reductive, then $\Ker\hat\tau$ is always a free 
$\bbk[\g]$-module \cite[Theorem\,8.6]{p05}.
There is also a special case in which all the assumptions of Proposition~\ref{prop-c}
are satisfied.

\begin{thm}  \label{thm-c}
Let $\g$ be reductive, $\te\subset\g$ a Cartan subalgebra, and $e\in \g$ a regular nilpotent element. Suppose that\\
\hbox to \textwidth{\ $(\Diamond)$ \hfil 
 $\dim V^e=\dim V^\te$. \hfil }
 \\[.7ex]
Then $\Omega_\tau\supset\g_{reg}$ and if 
$\eus F(x)\in x{\cdot}V$ for all $x\in\g_{reg}$, then there is
$F:\g\to V$ such that $\eus F(x)=x{\cdot}F(x)$ for all $x\in \g$.
\end{thm}\begin{proof}
It easily follows from assumption $(\Diamond)$ that 
$\dim V^z=\dim V^{\te}$ for any regular  semisimple $z\in\g$.
Therefore the minimal value of $\dim V^x$ is the dimension of the zero-weight
space of $V$, which is positive. 
That is, the open subset $\Omega_\tau$ contains the regular semisimple and nilpotent
elements. It follows that $\Omega_\tau\supset\g_{reg}$.
By \cite{ko63}, $\codim(\g\setminus \g_{reg})= 3$. For triviality $E'_\tau$, it is enough
to prove that
$\Ker\hat\tau$ is a free $\bbk[\g]$-module, and the latter has been done in 
\cite[Theorem\,8.6]{p05}.
\end{proof}

\begin{rmk} 
The above equality $\dim V^z=\dim V^{\te}$ (with $z$ semisimple) means that each 
nonzero weight of
$V$ (with respect to $\te$) is a multiple of a root. Using this observation,
one easily  obtains
the complete list of irreducible representations of simple Lie algebras
satisfying assumption $(\Diamond)$.
Here it is:

\begin{itemize}
\item the adjoint representation of $\g$;
\item $(\GR{B}{n}, \vp_1)$,  $(\GR{B}{n}, 2\vp_1)$, $(\GR{C}{n}, \vp_2)$, 
$(\GR{F}{4}, \vp_1)$, $(\GR{G}{2}, \vp_1)$, $(\GR{G}{2}, 2\vp_1)$;
\item $(\GR{A}{1}, 2m\vp_1)$ for any $m\in\BN$.
\end{itemize}
\end{rmk}

\noindent
Actually, each of  the cases (A), or (B), or (C) deserves a special thorough treatment.
In the following sections, we concentrate on case (A), partly in view of its connections with
differential operators. Another reason is that similar properties of sequences (B) and (C)
for representations of reductive groups have been studied in \cite[Section\,8]{p05}.
                 
\section{Differential operators and invariant polynomials} 
\label{sect:diff_op}
                                                    
\noindent  
In Section~\ref{sect:mod}, three possibilities to generalise Dixmier's results
have been discussed. These are related to three sequences of modules over polynomial rings.
It seems that case~(A) is the most interesting one, because the problem can further be 
transferred to the setting of differential operators on $V$.

The discussion of case (A) in Section~\ref{sect:mod} shows that if a $G$-module $V$
satisfies certain explicit conditions, then
a vector field $\eus F:V\to V$ annihilates all of\/
$\bbk[V]^G$ if and only if there is $F\in \Mor(V,\g)$ such that $\eus F(v)=F(v){\cdot}v$
for all $v\in V$.
In other words,  \vskip.7ex
\hbox to \textwidth{\ $(\lozenge)$ \hfil
$\eus F\{f\}=0$ for any $f\in \bbk[V]^G$ if and only if $\eus F\in \bbk[V]\boldsymbol{\varsigma}(\g)$.
\hfil}

\noindent (cf. Remark~\ref{istolkov}).
Let us restate $(\lozenge)$ using the algebra of differential operators $\eus D(V)$.

Let $\eus C=\mathsf{Cent}_{\eus D(V)}(\bbk[V]^G)$ denote the centraliser of $\bbk[V]^G$ 
in $\eus D(V)$.
Clearly, $\eus C$ contains $\bbk[V]$ and $\boldsymbol{\varsigma}(\g)$. 
Let $\eus A$ be the subalgebra of $\eus C$ generated by $\bbk[V]$ and $\boldsymbol{\varsigma}(\g)$.
Note that  a vector field $\eus F$ and a polynomial $f\in\bbk[\g]^G$
commute as differential operators 
if and only if $\eus F\{f\}=0$. 
Therefore assertion $(\lozenge)$ can also be interpreted as the coincidence
of $\eus A$ and $\eus C$ on the level of vector fields. 

Motivated by Dixmier's result \cite[Theorem\,2.1]{dixm} and a question by Barlet, 
Levasseur and Stafford proved that $\eus A=\eus C$ for the adjoint 
representation of a semisimple Lie algebra $\g$ \cite{ls}. 
In this section, we prove such an equality in a more general setting.

We  assume below that the $G$-module $V$ has the property that $\bbk[V]^G$ is 
finitely generated and  the quotient field of 
$\bbk[V]^G$ equals $\bbk(V)^G$. The latter is equivalent to that a generic fibre of
$\pi_V: V\to V\md G$ contains a dense $G$-orbit.
 
We work with the sequence of graded $\bbk[V]$-modules
\[
0\to\Ker\hat\phi\to \Mor(V,\g)\stackrel{\hat\phi}{\to} \Mor(V,V) . 
\]
Here $\rk\hat\phi=\max_{v\in V}\dim \g{\cdot}v$ \cite[Prop.\,1.7]{jac} and therefore
$\rk(\Ker\hat\phi)=\min_{v\in V}\dim\g_v$.

\noindent
Set $m=\min_{v\in V}\dim\g_v=\dim\g -\dim V+\dim V\md G$.
Assume that $\Ker\hat\phi$ is a free (graded) $\bbk[V]$-module, 
and let $F_1,\dots, F_m$ be a homogeneous basis for $\Ker\hat\phi$.
Write $E$ for the $\bbk[V]$-module $\Ima\hat\phi$. Then we obtain the exact sequence
\begin{equation}  \label{eq:exact}
 0\to \bigoplus_{i=1}^m \bbk[V]\,F_i \stackrel{\hat\beta}
\to\Mor(V,\g)\stackrel{\hat\phi}{\to} E\to 0 \ .
\end{equation}
\noindent
Using the morphisms $F_i:V\to \g$, we define a variety $Y$ as follows:
\[
    Y=\{(v,\eta)\in V\times\g^* \mid \langle F_i(v),\eta\rangle=0, \ i=1,\dots,m\}.
\]
Recall that $\Ker\hat\phi$ is a $G$-stable submodule of $\Mor(V,\g)$. Therefore
for any $g\in G$ there exist $u_1^{(g)},\dots,u_m^{(g)}\in \bbk[V]$ such that
$g\ast F_i=\sum_{j} u_j^{(g)} F_j$. This readily implies that $Y$ is $G$-stable.

\noindent
Choose a basis $e_1,\ldots,e_n$ for $\g$. Using this basis, we identify
$\Mor(V,\g)=\bbk[V]\otimes \g$ with $\bbk[V]^n$.
Then we can write $F_j(v)=\sum_{i=1}^n F_{ij}(v)e_i$, where 
$F_{ij}\in \bbk[V]$.
If we regard~\eqref{eq:exact} as a sequence
\[
    0\to\bbk[V]^m\stackrel{\hat\beta}{\to} \bbk[V]^n\stackrel{\hat\phi}{\to} E \to 0 \ ,
\]
then  $\hat\beta$ becomes an $n\times m$-matrix with entries $F_{ij}$.
Let $I_t(\hat\beta)$ be the ideal of $\bbk[V]$ generated by $t\times t$ minors of $\hat\beta$.
Following \cite{hsv}, consider the series of determinantal conditions
for the presentation of $E$: 
\\[1ex]
\hbox to \textwidth{\quad 
$(\mathcal F_d)$\hfil $\hot I_t(\hat\beta)\ge m-t+1+d$ \ \ \ \text{for } $1\le t\le m$. \hfil}

\vskip.7ex\noindent
The ideals $I_t(\hat\beta)$ are independent of the presentation of $E$. These
are  {\it Fitting ideals\/} of $E$, see e.g. \cite[1.1]{vasc}.
Let $\mathsf{Sym}_{\bbk[V]}(E)$ denote the symmetric algebra of 
the $\bbk[V]$-module $E$.  Then 
$ \mathsf{Sym}_{\bbk[V]}(E)=\bigoplus_{n=0}^\infty \mathsf{Sym}_{\bbk[V]}(E)_n$
and each $\mathsf{Sym}_{\bbk[V]}(E)_n$ is a finitely generated graded $\bbk[V]$-module.
Conditions $(\mathcal F_d)$ are very useful in the study of properties of the 
symmetric algebras of modules.  Utility of these conditions in Representation and 
Invariant theory has been demonstrated in \cite{jac, ls, p05}.

\begin{thm}   \label{thm:main4-1}
Suppose that $\Ker\hat\phi$ is a free module and condition $(\mathcal F_2)$ is 
satisfied by $E$. Then
\begin{itemize}
\item[\sf (i)] \ $\mathsf{Sym}_{\bbk[V]}(E)$ is a factorial domain of Krull dimension
$\dim V+\dim\g-m$.
\item[\sf (ii)] \ $Y$ is an irreducible
 factorial complete intersection, and\/ $\bbk[Y]=\mathsf{Sym}_{\bbk[V]}E$.
\item[\sf (iii)] \ $Y=\ov{\Ima(\varkappa)}$, where 
$\varkappa:  V\times V^*\to V\times \g^*$ is defined by 
$\varkappa(v,\xi)=(v,\mu(v,\xi))$. Here $v\in V, \xi\in V^*$ and $\mu: V\times V^* \to \g^*$ is the moment mapping.
\item[\sf (iv)] \ Let $p:Y\to V$ be the (surjective) projection. If $\mathfrak I$ is a prime ideal of
$\bbk[V]$ with $\hot(\mathfrak I)\ge 2$, then $\hot(\mathfrak I\,\bbk[Y])\ge 2$.
\end{itemize}
\end{thm}\begin{proof}
(i) \ The exact sequence \eqref{eq:exact} shows that $E$ has projective dimension at most
one. Therefore part (i) follows from 
\cite[Prop. 3 \& 6]{luch}.

(ii) \ The universal property of symmetric algebras implies that
$\Sym_{\bbk[V]}(E)$ is the quotient of $\Sym_{\bbk[V]}(\bbk[V]\otimes \g)=
\bbk[V\times \g^*]$ modulo the ideal "generated by the image of $\hat\beta$". More precisely,
each $F_i$ determines the polynomial $\widehat F_i\in \bbk[V\times\g^*]$ by the rule
$\widehat F_i(v,\eta)=\langle F_i(v),\eta\rangle$, $\eta\in\g^*$,  and the ideal in question
is generated by $\widehat F_1,\dots,\widehat F_m$.
Hence $\Sym_{\bbk[V]}(E)=\bbk[Y]$, and the other assertions follow from (i).

(iii) \ Clearly, $\ov{\Ima(\varkappa)}$ is an irreducible subvariety of $V\times \g^*$.
Taking the (surjective) projection to $V$ and looking at the dimension of the 
generic fibre, one finds that $\dim \ov{\Ima(\varkappa)}=\dim V+\max(\dim \g{\cdot}v)=
\dim V+\dim\g-m$. Since $F_i(v){\cdot}v=0$ for all $v\in V$, we have
\[
  \widehat F_i(v,\mu(v,\xi))=\langle F_i(v), \mu(v,\xi)\rangle= \langle F_i(v){\cdot}v, \xi\rangle= 0 \ .
\] 
Hence each  $\widehat F_i$ vanishes on $\ov{\Ima\varkappa}$
and  $\ov{\Ima(\varkappa)}\subset Y$.
Since both varieties have the same dimension and are 
irreducible, they are equal.

(iv) \ According to \cite{hsv}, Remarks on pp. 664--5, this property  is equivalent to 
condition $(\mathcal F_2)$.   See also \cite[Remark\,1.3.9]{vasc}
\end{proof}

As a by-product of this theorem, we obtain the following description of 
$\Sym_{\bbk[V]}(E)=\bbk[Y]$.  Consider the linear map 
$\mathsf{i}:\g\to V^*\otimes V$ which is induced by the  moment map 
$\mu: V\times V^*\to\g^*$. (The map $\mathsf{i}$ is injective if and only if the representation
$\g\to \mathfrak{gl}(V)$ is faithful.)
In this way, we obtain certain copy
of $\g$ sitting in  $V^*\otimes V\subset \bbk[V\times V^*]_2$. 
Then $\bbk[Y]$ is isomorphic to the subalgebra of $\bbk[V\times V^*]$
generated by $\bbk[V]$ and $\mathsf i(\g)$.

\begin{thm}   \label{thm:main4-2}
Suppose that 
\begin{itemize}
\item[\sf 1.]  \ $\bbk[V]^G$ is a polynomial algebra freely generated by $f_1,\dots,f_l$;
\item[\sf 2.]  \ $\Ker\hat\phi$ is a free $\bbk[V]$-module;
\item[\sf 3.]  \   $V$ has the codim--$2$ property and  $\{(\textsl{d}f_i)_v\}$ are 
linearly independent for any $v\in V_{reg}$;
\item[\sf 4.]  \  condition $(\mathcal F_2)$ is satisfied for $E=\Ima\hat\phi$.
\end{itemize}
Let $\eus A$ and $\eus C$ be given the filtration induced from $\ml D(V)$.
Then $\gr\eus C= \gr\eus A \simeq \Sym_{\bbk[V]}(E)$.
\end{thm}\begin{proof}
The proof  of Levasseur and Stafford  for the adjoint representation of
a semisimple $\g$  \cite[Section~3]{ls}
carries over {\sl mutatis mutandis\/} to this more general situation. 
The following is very close to their original proof.

\begin{lm}    \label{lm:3.1}
For $v\in V_{reg}$, let $\eus R$ denote the local ring of $V$ at $v$. Then there exists a basis 
of derivations $\{\partial_i \mid i=1,\dots,s\}$ of\/ $\Der\eus R$ such that 
$\partial_i(f_j)=\delta_{i,j}$ for all $1\le i,j\le l$ and $\eus R\boldsymbol{\varsigma}(\g)=
\bigoplus_{i=l+1}^s \eus R\partial_i$. \quad  \normalfont{[Here $s=\dim V$.]}
\end{lm}
\begin{proof}
There is a one-to-one correspondence between the bases for the $\eus R$-module
$\Der\eus R$ and the $\bbk$-bases for the tangent space $T_v(V)$. 
Since $\{(\textsl{d}f_i)_v, \ i=1,\dots,l\}$ are linearly independent, 
the annihilator of $\mathsf{span}\{(\textsl{d}f_i)_v, \ i=1,\dots,l\}$ in
$T_v(V)$ is $\g{\cdot}v$. Choose a basis $(e_1,\dots,e_s)$ in $V$ such that 
$\langle e_i, (\textsl{d}f_j)_v\rangle=\delta_{i,j}$, $1\le i,j\le l$, and 
$\mathsf{span}\{e_{l+1},\dots,e_s\}=\g{\cdot}v$. Then the corresponding basis for
$\Der\eus R$ will work. (Cf. the proof of Lemma\,3.1 in \cite{ls}).
\end{proof}

Since $\gr\eus C\subset \gr\ml D(V)=\bbk[V\times V^*]$, certainly $\gr\eus A\subset\gr\eus C$
are domains. 
Next, $\gr\eus A$ contains the subalgebra generated by $\bbk[V]$ and $\gr(\boldsymbol{\varsigma}(\g))$.
It is easily seen that  $\gr(\boldsymbol{\varsigma}(\g))=\mathsf i(\g)$. It follows from Theorem~\ref{thm:main4-1} 
that, on the geometric level, we obtain the chain of dominant morphisms
\begin{equation}   \label{eq:chain}
   V\times V^* \to \spe(\gr\eus C)\to \spe(\gr\eus A)\to \ov{\Ima\varkappa}=
   Y\subset V\times\g^* \ ,
\end{equation}
where the resulting map $V\times V^* \to \ov{\Ima\varkappa}$ is just $\varkappa$.

\begin{lm}  \label{lm:3.2} The above mapping
$\vp: \spe(\gr\eus C)\to \ov{\Ima\varkappa}$ is birational.
\end{lm}\begin{proof}
We prove a more precise assertion that, for any $v\in V_{reg}$, there is the equality
of local rings $(\gr\eus C)_v=(\gr\eus A)_v=\bbk[Y]_v$.

Recall that $\eus R=\bbk[V]_v$.  For any ring $\mathfrak R$ containg $\bbk[V]$, 
its localisation with respect to the 
multiplicative subset $\{f\in\bbk[V]\mid f(v)\ne 0\}$ is denoted by $\mathfrak R_v$.
Then 
$ D(V)_v=\eus R\{ \partial_1,\dots,\partial_s\}$ is the non-commutative algebra generated 
by derivations constructed in Lemma~\ref{lm:3.1} and 
\[
\eus C_v=\{ D\in \ml D(V)_v \mid [D,f_i]=0, \ i=1,\dots,l\} .
\]
The last formula and Lemma~\ref{lm:3.1} readily imply that 
$\eus C_v=\eus A_v=\eus R\{ \partial_{l+1},\dots,\partial_s\}$. Let $\bar\partial_i$ denotes the
image of $\partial_i$ in $\gr\ml D(V)$. Lemma~\ref{lm:3.1} also shows that
\[
 \eus R{\cdot}\mathsf i(\g)=\bigoplus_{i=l+1}^s \eus R\,\bar\partial_i .
\]
Thus, 
$ \gr\eus C_v=\gr\eus A_v=\gr\eus R\{ \partial_{l+1},\dots,\partial_s\}=
    \eus R[ \bar\partial_{l+1},\dots,\bar\partial_s]_v=\eus R[\mathsf i(\g)]_v=\bbk[Y]_v$.
\end{proof}

\noindent 
Finally, we prove that $\vp: \spe(\gr\eus C)\to \ov{\Ima\varkappa}$ is an isomorphism.
We already know that $\vp$ is birational and that $\ov{\Ima\varkappa}$ is normal
(Theorem~\ref{thm:main4-1}). By Richardson's lemma, see e.g. \cite[3.2~Lemme\,1]{brion},
it suffices to verify that $\Ima\vp$ contains an open subset whose complement
is of codimension $\ge 2$. Thanks to Eq.~\eqref{eq:chain}, this reduces to the
same question for $\varkappa: V\times V^*\to \ov{\Ima\varkappa}\simeq Y$.

It is easily seen that if $(v,\xi)\in \ov{\Ima\varkappa}=Y$ and $v\in V_{reg}$,
then $(v,\xi)\in \Ima\varkappa$. Invoking the projection $p:Y\to V$ shows that 
$p^{-1}(V_{reg})$ is on open subset lying in $\Ima\varkappa$.
Since $V$ has the codim--$2$ property, we conclude, using Theorem~\ref{thm:main4-1}(iv),
that  the complement of $p^{-1}(V_{reg})$ is of codimension $\ge 2$.
\vskip.7ex
This completes the proof of Theorem~\ref{thm:main4-2}.
\end{proof}

\begin{cl}  
{\sf (i)} \ 
$\eus A=\eus C$; Moreover, $\eus C$ is an Auslander-Gorenstein, Cohen-Macaulay,
Noetherian domain and a maximal order;
{\sf (ii)} \  
the centre of\/ $\eus C$ is $\bbk[V]^G$.
\end{cl}
(See the proof of Corollary~3.3 in \cite{ls}.)

\noindent 
Levasseur and Stafford also prove that, in their situation, both $\eus C$ and $\gr\eus C=
\Sym_{\bbk[V]}(E)$ 
are free modules over $\bbk[\g]^G$, see \cite[Corollary~3.4]{ls}. We return to this question 
below.

There is a particular case of Theorem~\ref{thm:main4-2},
where the assumptions simplify considerably.
\begin{prop}   \label{prop:loc_free}
Suppose that 
\begin{itemize}
\item[\sf 1.]  \ $\bbk[V]^G$ is a polynomial algebra freely generated by $f_1,\dots,f_l$;
\item[\sf 2.]  \   $V$ has the codim--$2$ property and  $\{(\textsl{d}f_i)_v\}$ are 
linearly independent for any $v\in V_{reg}$;
\item[\sf 3.] \  there is $v\in V$ such that $\g_v=0$.
\end{itemize}
Then $Y=\ov{\Ima\varkappa}=V\times \g^*$ and $\eus A=\eus C$.
\end{prop}
\begin{proof}
Indeed, here $\Ker\hat\phi=0$ and condition $(\mathcal F_2)$ becomes 
vacuous. 
\end{proof}

\noindent
Verification of condition $(\mathcal F_2)$ is the most difficult part in possible applications
of Theorem~\ref{thm:main4-2}. 
In the rest of the section, we describe a  geometrical
approach to it (cf. similar approach in \cite[\S\,8]{p05}). Let us assume that 
the first two hypotheses of that theorem are satisfied.
In particular, $\Ker\hat\phi$ is a free  module (of rank $m$).

Using the basis morphisms $F_i:V\to\g$, define the stratification of $V$ as follows:
\[
   {\eus X}_i =\{ v\in V\mid \dim \mathsf{span}\{F_1(v),\ldots,
F_m(v)\}\le m-i\} \ .
\]
Then $\eus X_{i+1}\subset \eus X_{i}$ and $\eus X_0=V$. As the ideal
$I_t(\hat\beta)$ defines $\eus X_{m-t+1}$, condition $(\mathcal F_2)$ precisely
means that $\dim\eus X_i\le \dim V-i-2$ for any $i\ge 1$. In particular,
$\codim\, \eus X_{1}\ge 3$. In case of the coadjoint representation of a $3$-wonderful
Lie algebra, this becomes just the {\sl codim}--$3$ condition on the set of non-regular points,
which is used in the proof of Theorem~\ref{ala}.

Consider the quotient map $\pi_V: V\to V\md G\simeq \mathbb A^l$.
As usual, we say that $\pi_V^{-1}(\pi_V(0))=:\eus N=\eus N_V$ 
is the {\it null-cone\/} of (the $G$-module) $V$. Set $\eus X_i(\eus N):=\eus X_i\cap\eus N$.
Sometimes, the study of $\{\eus X_i\}$ can be reduced to that of $\{\eus X_i(\eus N)\}$. 

Recall that $\Ker\hat\phi$ is a $G$-stable submodule of $\Mor(V,\g)$. Therefore if
$F_1,\dots,F_m$ is a basis of $\Ker\hat\phi$, then $\{g\ast F_i\}_i$ is another basis
for any $g\in G$. It is not clear a priori that the $F_i$'s should be
$G$-equivariant. Consequently, subvarieties $\eus X_i$ are not necessarily $G$-stable.
However, in all known examples the freeness of $\Ker\hat\phi$ does mean that there is a basis
consisting of $G$-equivariant morphisms. (Cf. Remark~\ref{essence} and Theorem~\ref{thm:z2-free} below). For this reason, we wish to assume that $F_i\in \Mor_G(V,\g)$.

\begin{prop}    \label{pr:visib}
Under the first two assumptions of Theorem~\ref{thm:main4-2}, 
suppose that  a generic fibre of $\pi_V$ is a (closed) $G$-orbit,  
$\eus N$ contains finitely many $G$-orbits, and each $F_i$ lies in $\Mor_G(V,\g)$.  If 
\\[.6ex]
\hbox to \textwidth{\ 
$(\clubsuit)$ \hfil $\codim_{\eus N} \eus X_i(\eus N) \ge i+1$ for any \ $i\ge 1$, \hfil }
\\[.6ex]
then $(\mathcal F_2)$ is satisfied. 
\\ An equivalent form of
condition $(\clubsuit)$ is that  $\eus N_{reg}\subset \eus X_0(\eus N)$ and, 
for any 
$v\in \eus N\setminus \eus N_{reg}$, 
\[  
    \dim(\mathsf{span}\{ F_i(v) \mid i=1,\dots,m \})+\codim_{\eus N} G{\cdot}v \ge m+1 .
\]  
\end{prop}
\begin{proof} The finiteness assumption guarantees us that $\dim\eus N=\dim V-\dim V\md G$
and $\eus N_{reg}\subset V_{reg}$.  Furthermore, all the fibres of $\pi_V$ are of dimension 
$\dim V-\dim V\md G$, see e.g. \cite[Prop.\,6 in p.146]{brion} 
(the reductivity of $G$ is not needed in this place). By Chevalley's theorem, $\pi_V$ is an
open map. Consequently, it is onto.
By the assumption, each $\eus X_i$ is $G$-stable.
Let $\pi_{i,V}$ be the restriction of $\pi_V$ to $\X_i$. Then
$\eus X_i(\eus N)=\pi_{i,V}^{-1}(\pi_{i,V}(0))$. 
Therefore 
\[
    \dim\eus X_i\le \dim \eus X_i(\eus N)+ \dim \pi_V(\eus X_i)\le
    \dim \eus N-(i+1) + (\dim V\md G -1)=\dim V - (i+2).
\]
Here we used the fact that $\ov{ \pi_V(\eus X_i)}$ is a proper subvariety of $V\md G$ 
for $i\ge 1$. Indeed, $V\setminus \eus X_1$ is a dense open subset of $V$ and 
there is a dense open subset $\Xi \subset (V\setminus \eus X_1)$ such that
if $G{\cdot}v$ is a  fibre of $\pi_V$ for any $v\in\Xi$. 
The second part is an easy reformulation of condition $(\clubsuit)$, which uses the finiteness
for $G$-orbits in $\eus N$.
\end{proof}

\noindent
Recall that $v\in V_{reg}$ if and only if $\dim\g_v=m$.
Since $\codim_{\eus N}G{\cdot}v  =\dim \g_v -m$, yet another form of 
the above conditions is 
\begin{gather}  \label{eq:F3-orb}
  \dim(\mathsf{span}\{ F_i(v) \mid i=1,\dots,m \})+\dim\g_v = 2m   \text{ for any } \ 
  v\in \eus N_{reg},  \\  \label{eq:F4-orb}
  \dim(\mathsf{span}\{ F_i(v) \mid i=1,\dots,m \})+\dim\g_v \ge 2m+1   \text{ for any } \ 
  v\in\eus N\setminus \eus N_{reg} .
\end{gather}

\begin{rmk}   \label{rmk:sgp-abelian}
If $\Ker\hat\phi$ is a free $\bbk[V]$-module generated by $G$-equivariant morphisms, then
a generic stabiliser for $(G:V)$ is commutative. Indeed, the $G$-equivariance implies
that $F_i(v)$ lies in the centre of $\g_v$ for any $v\in V$.
On the other hand, if $v$ is generic, then $F_1(v),\dots,F_m(v)$ form a basis for $\g_v$.
(Cf. Remark~3.2 in \cite{coadj}.)
\end{rmk}

The above inequality $(\clubsuit)$ is crucial for establishing $(\mathcal F_2)$ in  
applications. Furthermore,  it essentially implies that 
$\eus C$ to be a free $\bbk[V]^G$-module.

\begin{thm}   \label{thm:C_is_free} Suppose that
\begin{itemize}
\item[\sf 1.]  \ $\bbk[V]^G$ is a polynomial algebra freely generated by $f_1,\dots,f_l$;
\item[\sf 2.]  \   $V$ has the codim--$2$ property and  $\{(\textsl{d}f_i)_v\}$ are 
linearly independent for any $v\in V_{reg}$;
\item[\sf 3.]  \ $\Ker\hat\phi$ is a free $\bbk[V]$-module generated $G$-equivariant morphisms
$F_1,\dots,F_m$;
\item[\sf 4.]  \  a generic fibre of $\pi_V$ is a (closed) $G$-orbit;  
\item[\sf 5.]  \  $\eus N$ contains finitely many $G$-orbits;
\item[\sf 6.]  \   $\codim_{\eus N} \eus X_i(\eus N) \ge i+1$ for any \ $i\ge 1$.
\end{itemize}
Then $\gr\eus A=\gr\eus C=\mathsf{Sym}_{\bbk[V]}(\Ima\hat\phi)$, $\eus A=\eus C$, and 
both $\eus C$ and $\gr\eus C$ are free (left or right) $\bbk[V]^G$-modules.
\end{thm}
\begin{proof}
As the hypotheses imply condition $(\mathcal F_2)$ for $E=\Ima\hat\phi$, only the last assertion
requires a proof.

Recall that  $\bbk[Y]=\gr\eus C$ and 
$Y=\ov{\Ima(\varkappa)}$ is a complete intersection  of codimension $m$ in 
$V\times\g^*$.
Consider the map $\nu: Y\to V\to V\md G\simeq \mathbb A^l$,
the composition of the projection and the quotient morphism $\pi_{V}$.
Its fibre over the origin is 
\[
\eus Z:=\{(v,\xi) \mid v\in \eus N \ \ \& \ \   \langle F_i(v), \xi\rangle=0, \ i=1,\dots,m\} .
\]
We wish to prove that $\eus Z$ is a variety of pure dimension  $\dim Y- l$.  
Since $\eus Z$ is a fibre
of a dominant map $Y\to \mathbb A^l$,  we have $\dim\eus Z\ge \dim Y-l$.
On the other hand, consider the projection $p: \eus Z \to \eus N$.  It follows from 
hypothesis {\sf 6} that  
\[
\dim p^{-1}( \eus X_i(\eus N)\setminus  \eus X_{i+1}(\eus N)) \le \dim\eus N-i -1 + \dim\g -m+i=
 \dim Y- l -1 
\]  
for any $i\ge 1$.  Hence $\eus Z=\ov {p^{-1}(\eus N\setminus  \eus X_{1}(\eus N))}$ is of pure dimension $\dim Y-l$. [Furthermore, it is not hard to prove that $p$ provides a one-to-one 
correspondence between 
the irreducible components of $\eus Z$ and $\eus N$.]
Consequently, $\nu$ is equidimensional.
Since $Y$ is Cohen-Macaulay, the $\nu$ is also flat. This shows that each 
$(\gr\eus C)_n$ is a flat graded finitely generated module over the polynomial ring
$\bbk[V]^G$, hence a free module. Thus, $\gr\eus C$ is a free $\bbk[V]^G$-module.

The assertion on $\eus C$ can be proved exactly as in Corollary\,3.4 in \cite{ls}.
\end{proof}

\begin{rmk}
Conditions $(\mathcal F_d)$ can also be considered for $\Ima\hat\psi$ or 
$\Ima\hat\tau$  whenever
$\Ker\hat\psi$ is a free $\bbk[V]$-module or $\Ker\hat\tau$ is a free $\bbk[\g]$-module.
If $(\mathcal F_2)$ is satisfied, then one obtains a similar description for the image of
$\varkappa_\psi : V\times \g\to V\times V$, $(v,x)\mapsto (v, x{\cdot}v)$ or
$\varkappa_\tau : \g\times V\to \g\times V$, $(x,v)\mapsto (x, x{\cdot}v)$, see 
Theorems\,8.8 and 8.11 in \cite{p05}. However such descriptions seem to have no 
non-commutative counterparts.
\end{rmk}

\section{Applications: isotropy representations of symmetric pairs and beyond}  
\label{sect:z2}

\noindent 
In this section,  $G$ is a connected {\sl semisimple\/} algebraic group.
If $\g=\g_0\oplus\g_1$ is a 
$\BZ_2$-grading of $\g$, then $(\g,\g_0)$ is said to be a {\it symmetric pair}.
Let $G_0$ be the connected subgroup of 
$G$ with Lie algebra $\g_0$. Our goal is to describe a class of $\BZ_2$-gradings
that lead to isotropy representations $(G_0:\g_1)$ satisfying the assumptions of 
Theorems~\ref{thm:main4-2} and \ref{thm:C_is_free}.

\noindent  Recall necessary results on 
the representation $(G_0:\g_1)$. The standard reference for this is \cite{kr71}.
Let $\N$ denote the set of nilpotent elements of $\g$. 

\begin{itemize}
\item[$\langle 1\rangle$]\  Any $v\in\g_1$ admits a unique decomposition $v=v_s+v_n$, where $v_s\in\g_1$ is 
semisimple and $v_n\in \N\cap\g_1$; $v=v_s$ if and only if $G_0{\cdot}v$ is closed;
$v=v_n$ if and only if the closure of $G_0{\cdot}v$ contains the origin.
For any $v\in\g_1$, there is the induced $\BZ_2$-grading of the centraliser 
\ $\g_v=\g_{0,v}\oplus \g_{1,v}$, and $\dim\g_0-\dim\g_1=\dim\g_{0,v}-\dim\g_{1,v}$.

\item[$\langle 2\rangle$]\  Let $\ce\subset\g_1$  be a maximal subspace consisting of pairwise commuting
semisimple elements. 
All such subspaces are $G_0$-conjugate and $G_0{\cdot}\ce$ is dense in 
$\g_1$; $\dim\ce$ is called the {\it rank\/} of the $\BZ_2$-grading or pair $(\g,\g_0)$, 
denoted $\rk(\g,\g_0)$.  
If $v\in\ce\cap(\g_1)_{reg}$, then $\g_{1,v}=\ce$ and $\g_{0,v}$ is a generic stabiliser
for the $G_0$-module $\g_1$.

\item[$\langle 3\rangle$]\  The algebra $\bbk[\g_1]^{G_0}$ is polynomial and $\dim\g_1\md G_0=\rk(\g,\g_0)$. 
The quotient map $\pi_{\g_1}: \g_1\to \g_1\md G_0$ is equidimensional. The generic fibre is
the orbit of a  $G_0$-regular semisimple element.
Each fibre of $\pi_{\g_1}$ contains finitely many $G_0$-orbits and each closed $G_0$-orbit in
$\g_1$ meets $\ce$. The null-cone in $\g_1$ equals $\N\cap\g_1$.

\item[$\langle 4\rangle$] \ If $v\in (\g_1)_{reg}$ and $f_1,\dots,f_{\dim\ce}\in \bbk[\g_1]^{G_0}$ are basis invariants, 
then the $\{(\textsl{d}f_i)_v\}_i$ are linearly independent.
\end{itemize}
A $\BZ_2$-grading (or a symmetric pair $(\g,\g_0)$) is said to be $\N$-{\it regular\/} if $\g_1$ 
contains a regular nilpotent element of $\g$. By a result of Antonyan~\cite{an}, 
a $\BZ_2$-grading is  $\N$-regular if and only if $\g_1$ contains a regular
semisimple element. Then $\dim\g_0-\dim\g_1=\rk\g-2\dim\ce$. 
\\ Now we are interested in properties  of 
\begin{equation}     \label{ed:A_g1}
0\to \Ker\hat\phi \to \Mor(\g_1,\g_0)\overset{\hat\phi}{\to} \Mor(\g_1,\g_1)  
\end{equation}
for $\N$-regular $\BZ_2$-gradings. Item $\langle 2\rangle$ 
above shows that a generic stabiliser for $(G_0{:}\g_1)$ is commutative
if and only if $\ce$ contains a regular semisimple element. Therefore
a $\BZ_2$-grading
is $\N$-regular if and only if a generic stabiliser for $(G_0{:}\g_1)$ is commutative.
In view of Remark~\ref{rmk:sgp-abelian}, these are the only $\BZ_2$-gradings, where 
one could expect that $\Ker\hat\phi$ is generated by $G_0$-equivariant morphisms.
The following  is  \cite[Theorem\,5.8]{coadj}.

\begin{thm}    \label{thm:z2-free}
Suppose the $\BZ_2$-grading $\g=\g_0\oplus\g_1$ is $\N$-regular. 
Then 
$\Ker\hat\phi$ is a free $\bbk[\g_1]$-module
generated by $G_0$-equivariant morphisms. In this case
$\rk(\Ker\hat\phi)=\rk\g-\dim\ce$. 
\end{thm}

\noindent 
Recall the construction of such a basis for $\Ker\hat\phi$. By \cite[Theorem\,4.7]{theta},
$Z=\ov{G{\cdot}\g_1}$ is a normal complete intersection in $\g$, $\codim\, Z=\rk\g-\dim\ce$,
and the ideal of $Z$ in $\bbk[\g]$ is generated by certain homogeneous 
basis invariants in $\bbk[\g]^G$. 
Let $f_1,\dots,f_m$ be such basis invariants ($m=\rk\g-\dim\ce$). 
As in Section~\ref{sect:diff_op}, any $F_i\in \Ker\hat\phi$ determines the polynomial
$\widehat F_i\in\bbk[\g_1\times\g_0^*]= \bbk[\g_1\times\g_0]\simeq\bbk[\g]$ and vice versa.
In this case, $\widehat F_i$ is defined to be the bi-homogeneous component of
$f_i$ of degree $(1,\deg f_i{-}1)$. (Here "1" is the degree with respect to $\g_0$.)
Since $\widehat F_i$ is clearly $G_0$-invariant, $F_i$ is $G_0$-equivariant.
From this description, it follows  that $(\textsl{d}f_i)_v=F_i(v)$ if $v\in\g_1$.
\\  \indent
For $\g$ simple, the list of $\N$-regular symmetric pairs consists of
symmetric pairs of maximal rank (when $\dim\ce=\rk\g$ and hence $\Ker\hat\phi=0$) 
and the following 4 cases:

\begin{tabbing}
\framebox{3} \= \quad  $(\mathfrak{so}_{2n}, \mathfrak{so}_{n-1}\dotplus\mathfrak{so}_{n+1})$, 
$n\ge 4$,\quad  \=  $m=\rk(\Ker\hat\phi)=1$;  \kill
\framebox{1} \> \quad $(\mathfrak{sl}_{2n}, \mathfrak{sl}_n\dotplus\mathfrak{sl}_n\dotplus \te_1)$
 \>
$m=\rk(\Ker\hat\phi)=n{-}1$;
\\
\framebox{2} \>  \quad $(\mathfrak{sl}_{2n+1}, \mathfrak{sl}_n\dotplus\mathfrak{sl}_{n+1}\dotplus \te_1)$  \>
$m=\rk(\Ker\hat\phi)=n$;
\\
\framebox{3} \> \quad  $(\mathfrak{so}_{2n}, \mathfrak{so}_{n-1}\dotplus\mathfrak{so}_{n+1})$, 
$n\ge 4$  \>
$m=\rk(\Ker\hat\phi)=1$;
\\
\framebox{4} \>  \quad $(\GR{E}{6}, \GR{A}{5}\times \GR{A}{1})$ \>  $m=\rk(\Ker\hat\phi)=2$.
\end{tabbing}
\begin{thm}     \label{thm:z2-F2}
Suppose the $\BZ_2$-grading $\g=\g_0\oplus\g_1$ is $\N$-regular, and
$(\g,\g_0)\ne (\mathfrak{sl}_{2n+1}, \mathfrak{sl}_n\dotplus\mathfrak{sl}_{n+1}\dotplus \te_1)$.
Then  inequality $(\clubsuit)$ holds for $\eus N_{\g_1}=\N\cap\g_1$ and therefore 
$(\mathcal F_2)$ is satisfied for\/ $\Ima\hat\phi$ in Eq.~\eqref{ed:A_g1}.
\end{thm}\begin{proof} 
In the maximal rank case, $\Ker\hat\phi=0$ and therefore condition $(\mathcal F_2)$ is
vacuous.  In the remaining four cases, we have to resort to explicit calculations.
By results of \cite{kr71},  Proposition~\ref{pr:visib} applies in this situation. Hence our goal is
to verify whether  Eq.~\eqref{eq:F3-orb} and \eqref{eq:F4-orb} are satisfied for 
the elements of $\N\cap\g_1$.

Let $F_1,\dots,F_m$ be a basis for $\Ker\hat\phi$. 
The above (classification-free) construction of the $F_i$'s implies that 
$\dim(\mathsf{span}\{F_1(v),\dots,F_m(v)\})=m$ for any $v\in (\g_1)_{reg}$
(see the beginning of Section\,5 in \cite{coadj}). Therefore  Eq.~\eqref{eq:F3-orb}
is true in all four cases. It remains to handle Eq.~\eqref{eq:F4-orb}.

In the $\N$-regular case, each nilpotent $G$-orbit meets $\g_1$ \cite{an}. 
Therefore we can argue in terms of  nilpotent $G$-orbits in $\g$. 
Consider all the cases in turn.

For the first two cases, the explicit form of the  $F_i$'s  is pointed out in 
\cite[Example\,5.6]{coadj}. Namely, regarding elements  $v\in \g_1$ as  matrices 
(of order $2n$ or $2n+1$), we set $F_i(v)=v^{2i}$, the usual matrix power with $1\le i\le m$.

\begin{rem}
Strictly speaking, this formula for $F_i$ is only valid if the big Lie algebra is $\mathfrak{gl}_N$.
For  $\mathfrak{sl}_N$, one have to add a correcting term in order to ensure zero trace:
$F_i(v)=v^{2i}- (\tr(v^{2i})/N)I$. However, the correcting term vanishes if $v$ is nilpotent,
which is the case below.
\end{rem}

{\sf No.\,\framebox{1}} \ 
Let $(\eta_1,\eta_2,\ldots)$ be the  partition of $2n$ corresponding
to $v\in \N\cap\g_1$. Then $v^{2i}\ne 0$ if and only if $2i\le \eta_1-1$, and the nonzero 
terms are easily seen to be linearly independent.
Therefore $\dim(\mathsf{span}\{F_1(v),\dots,F_m(v)\})=\left\lfloor\frac{\eta_1-1}{2}\right\rfloor$.
Write $(\hat\eta_1,\hat\eta_2,\ldots,\hat\eta_s)$ for the dual partition. 
Then $s=\eta_1$. 
The term $\g_v$ occurring  in Eq.~\eqref{eq:F4-orb} now becomes $\g_{0,v}$.
The general equality $\dim G{\cdot}v=2\dim G_0{\cdot}v$ means in this case
that  $\dim\g_v=2\dim\g_{0,v}+1$. 
Because $\dim\g_v=\dim (\sltn)_v=\sum_{i=1}^s\hat\eta_i^2-1$,  the required inequality
looks as follows:
\[
\frac{1}{2}\sum_{i=1}^s\hat\eta_i^2+\left\lfloor\frac{\eta_1-1}{2}\right\rfloor-2n \ge 0 ,
\]
if $v\not\in (\N\cap\g_1)_{reg}$, i.e., if $\hat\eta_1\ge 2$. Since $\sum_{i=1}^s\hat\eta_i=2n$,
it is readily transformed into 
\[
    \frac{1}{2}\sum_{i=1}^s(\hat\eta_i-1)^2 +
\bigl(\left\lfloor\frac{s-1}{2}\right\rfloor-\frac{s}{2}\bigr) \ge 0 \ ,
\]
which is  true if $\hat\eta_1\ge 2$. Indeed, if $v$ is subregular, then $\hat\eta_1=2$,
$\hat\eta_j=1$ for $j\ge 2$, and $s=2n{-}1$. Hence the LHS is zero. For all other 
non-regular $SL_{2n}$-orbits the LHS is positive.

{\sf No.\,\framebox{2}} \ 
In this case, the numerical data are slightly different: $\sum_{i=1}^s\hat\eta_i=2n+1$,
$m=n$, and $\dim\g_v=2\dim\g_{0,v}$. The required inequality is
\[
    \frac{1}{2}\sum_{i=1}^s(\hat\eta_i-1)^2 +
\bigl(\left\lfloor\frac{s-1}{2}\right\rfloor-\frac{s+1}{2}\bigr) \ge 0 \ ,
\]
if $v\not\in (\N\cap\g_1)_{reg}$.  It fails to hold only if $v$ is subregular, i.e., 
$\hat\eta_1=2$, $\hat\eta_j=1$ for $j\ge 2$, and $s=2n$.
This means that $\codim_{\eus N}\eus X_1(\eus N)=1$ and 
$(\clubsuit)$ in Proposition~\ref{pr:visib} is not satisfied.
Using this, one can prove that condition $(\mathcal F_2)$ is not satisfied for
$\Ima\hat\phi$ \ here.

{\sf No.\,\framebox{3}}  \  
Since $F_1$ is the only  basis morphism,   the validity of Eq.~\eqref{eq:F4-orb} 
reduces to the assertion that $F_1(v)\ne 0$ whenever $G_0{\cdot}v$ is 
an orbit of codimension 1 in $\N\cap\g_1$.
The map $F_1$ arises from 
the pfaffian, \textsf{Pf}, a skew-symmetric matrix of order $2n$, and, as explained above, 
$F_1(u)=\textsl{d}(\mathsf{Pf})_u$.
If $u\in \mathfrak{so}_{2n}$ is nilpotent, then $\textsl{d}(\mathsf{Pf})_u= 0$
if and only if $u$ has at least three Jordan blocks \cite[Lemma\,4.4.1]{ri2}. 
However $G{\cdot}v\subset\N$ 
is the subregular nilpotent orbit and the corresponding partition is $(2n-1,3)$.

{\sf No.\,\framebox{4}}  \  Here $m=2$ and there are two basis morphisms in $\Ker\hat\phi$.
These two are associated with basis invariants of $G=\GR{E}{6}$
of degree 5 and 9. Therefore their degrees are equal to 4 and 8. Call them
$F^{(4)}$ and $F^{(8)}$, respectively.
Here the validity of Eq.~\eqref{eq:F4-orb} 
reduces to the assertions that 

$\begin{cases} \text{if $\codim_{\N\cap\g_1}G_0{\cdot}v=1$, then
$F^{(4)}(v)\ne 0$ and $F^{(8)}(v)\ne 0$,} \\
 \text{if $\codim_{\N\cap\g_1}G_0{\cdot}v=2$, then $F^{(4)}(v)\ne 0$ or $F^{(8)}(v)\ne 0$.}
 \end{cases}$

\noindent
In the first case, $G{\cdot}v$ is the subregular nilpotent orbit, usually denoted $\GR{E}{6}(a_1)$.
In the second case, $G{\cdot}v$ is the unique orbit of codimension 4 in $\N$,
denoted $\GR{D}{5}$.  The computations we need have been performed by Richardson,
see \cite[Appendix]{ri2}. He computed the "exponents" for all but one nilpotent orbits in the 
exceptional Lie algebras. In particular, 
for the $G$-orbit of type $\GR{E}{6}(a_1)$ (resp. $\GR{D}{5}$), the exponents include $4$ and 
$8$ (resp. include $4$).
This is exactly what  we need.
\noindent 
\end{proof}

It follows from the previous exposition that if $\g=\g_0\oplus\g_1$ 
is an $\N$-regular grading, then, modulo one exception, all the hypotheses
of Theorem~\ref{thm:C_is_free} 
are satisfied for the $G_0$-module $\g_1$.
Indeed, 
by above-mentioned  results of Kostant and Rallis \cite{kr71},
the hypotheses 1,\,2,\,4, and 5  hold for all $\BZ_2$-gradings.
The third  (resp. sixth) assumption is verified in  Theorem~\ref{thm:z2-free}
(resp. Theorem~\ref{thm:z2-F2}).

Thus, applying results  of Section~\ref{sect:diff_op}
to our  situation, we obtain
\begin{thm}    \label{thm:main5}
Suppose that $\g$ is simple, $\g=\g_0\oplus\g_1$ is an $\N$-regular $\BZ_2$-grading, and
$(\g,\g_0)\ne (\mathfrak{sl}_{2n+1}, \mathfrak{sl}_n\dotplus\mathfrak{sl}_{n+1}\dotplus \te_1)$.
Set $m=\rk\g-\dim\ce$. Then 
\begin{itemize}
\item[\sf (i)] \ $\mathsf{Sym}_{\bbk[\g_1]}(\Ima\hat\phi)$ is a factorial domain of Krull dimension
$\dim\g_1+\dim\g_0-m$.
\item[\sf (ii)] \  $\spe(\mathsf{Sym}_{\bbk[\g_1]}(\Ima\hat\phi))\simeq \ov{\Ima(\varkappa)}$, where
$\varkappa: \g_1\times\g_1\to \g_1\times\g_0$ is defined by $(v,v')\mapsto (v,[v,v'])$.
\item[\sf (iii)] \ $\ov{\Ima(\varkappa)}$ is an irreducible
 factorial complete intersection and its ideal is generated by $\widehat F_i$, $i=1,\dots,m$.
\item[\sf (iv)]\  $\eus A=\eus C$, i.e., $\eus C=\mathsf{Cent}_{\ml D(\g_1)}(\bbk[\g_1]^{G_0})$ is the algebra generated by $\bbk[\g_1]$ and $\boldsymbol{\varsigma}(\g_0)$.
\item[\sf (v)]\  Both $\eus C$ and $\gr\eus C=\mathsf{Sym}_{\bbk[\g_1]}(\Ima\hat\phi)$ are free 
(left or right) $\bbk[\g_1]^{G_0}$-modules.
\end{itemize}
\end{thm}

\begin{rmk}
For the symmetric pairs of maximal rank, the equality $\eus A=\eus C$
is proved in an unpublished manuscript of Levasseur \cite[Theorem\,4.4]{lev97}. In our 
exposition, this assertion also appears as a special case of Proposition~\ref{prop:loc_free}.
\end{rmk}

\begin{rmk}  \label{sl-2n+1}
As $(\mathcal F_2)$ is not satisfied for $\Ima\hat\phi$
in case of $(\mathfrak{sl}_{2n+1}, \mathfrak{sl}_n\dotplus\mathfrak{sl}_{n+1}\dotplus \te_1)$,
one might expect that some assertions of Theorem~\ref{thm:main5}
are wrong for that symmetric pair. However,
condition $(\mathcal F_1)$ still holds, and this is sufficient for proving that 
$Y=\ov{\Ima(\varkappa)}$ and it is a complete intersection whose ideal is generated 
by $\widehat F_1,\dots, \widehat F_m$.
It seems to be hard to check directly what 
is happening with assertion (iv). We are only able to prove that $\ov{\Ima(\varkappa)}$ \  
is \un{not} factorial for $(\mathfrak{sl}_3,\mathfrak{gl}_2)$. For, here $\ov{\Ima(\varkappa)}$ is 
a hypersurface in the 8-dimensional space $\g_1\times\g_0$, and the defining relation can 
be written up.
\end{rmk}

As Theorems~\ref{thm:main4-2} and  \ref{thm:C_is_free}
have rather general formulations (the group $G$ even is not 
supposed to be reductive!), it is instructive to have 
natural illustrations to it, which lie
outside the realm of the isotropy representations of symmetric pairs. In view of 
Proposition~\ref{prop:loc_free}, many representations with trivial generic stabiliser
will work. So, we concentrate on examples with non-trivial stabiliser, i.e.,
with $\Ker\hat\phi\ne 0$. 

\begin{ex}    \label{ex:theta3}
Take $G=SL_6\times SL_3$ and $V=\mathsf{R}(\varpi_2)\otimes \mathsf{R}(\varpi_1)$,
where $\varpi_i$ is the $i$-th fundamental weight.
This representation is associated with an 
automorphism of order 3 of $\GR{E}{7}$, and Vinberg's theory of $\theta$-groups \cite{vi76}, 
which is an
extension of the Kostant-Rallis theory, provides
a lot of information about it. In particular, $\bbk[V]^G$ is polynomial 
(with three generators) and $\eus N_V$ contains finitely many $G$-orbits.
Here $m=1$ and 
Proposition~\ref{pr:visib} is applicable.
The situation here resembles very much that  for $\N$-regular $\BZ_2$-gradings.
The basis covariant $F:V\to \g^*$ in $\Ker\hat\phi$ is associated
with the basis $\GR{E}{7}$-invariant $f$ of degree $10$. Therefore $\deg F=9$
and $F(v)=(\textsl{d}f)_v$ for $v\in V$. Since $m=1$, it suffices to verify
that $\dim \eus N_V-\dim \eus X_1(\eus N_V)\ge 2$. In other words, 
if $\co$ is a $G$-orbit of codimension 1 in $\eus N_V$, then we need $F\vert_\co\ne 0$.
An explicit classification of $G$-orbits in $\eus N_V$ \cite[\S\,4 Table~8]{gati} shows that 
orbits of codimension 1 lie inside of nilpotent $\GR{E}{7}$-orbits denoted by
$\GR{E}{7}(a_1)$ and $\GR{E}{7}(a_2)$.
Finally, using again Richardson's calculations \cite[Appendix]{ri2}, we obtain the required 
non-vanishing assertion.

We omit most details for this example, since we are going to consider applications of 
our theory to $\theta$-groups in a forthcoming article. 
\end{ex}

\begin{ex}    \label{ex:takif-F2}
We describe non-reductive Lie algebras whose coadjoint
representation satisfies the hypotheses of Theorem~\ref{thm:main4-2}.
This yields an incarnation of the ``$\eus A=\eus C$ phenomenon''
in the non-reductive case.
We detect such examples among $3$-wonderful algebras.
By a previous discussion 
(Remarks~\ref{rmk:3wond-conn} and \ref{essence}), hypotheses 1--3 are always 
satisfied for them. Thus, it remains only to have condition $(\mathcal F_2)$ for 
$\Ima\hat\phi$. Our examples exploit the semi-direct product construction (see  
Example~\ref{ex:new_wonder}). We start with $\es=\tri$ and set $\q=\tri\ltimes\tri$.
Then $\q$ is a quadratic $3$-wonderful algebra and $m=\ind\q=2$.
There are two basis $Q$-equivariant morphisms $F_i:\q^*\to \q$ in $\Ker\hat\phi$.
Identifying $\q^*$ and $\q$, we may regard $F_i$ as elements of $\Mor_Q(\q,\q)$.
Representing elements of $\q$ as pairs $(x,y)$, where $x,y\in\tri$, we obtain the following
explicit formulae: $F_1(x,y)=(x,y)$ and $F_2(x,y)=(0,x)$.
Then $\X_1=\{(0,y)\mid y\in\tri\}$ and $\X_2=\{(0,0)\}$. Hence condition $(\mathcal F_2)$
is satisfied. We have also checked $(\mathcal F_2)$ for 
$\mathfrak{sl}_3\ltimes\mathfrak{sl}_3$, $\mathfrak{sp}_4\ltimes\mathfrak{sp}_4$
and $\GR{G}{2} \ltimes\GR{G}{2}$.

Hopefully, this could be true if $\es$ is any simple Lie algebra, but we unable to prove it.
\end{ex}

\section{Some speculations} 
\label{sect:some_spec}

\noindent
There are two different  generalisations of Dixmier's result on adjoint vector fields in  the
context of the adjoint representation of semisimple Lie algebras.

One possibility is provided by the Levasseur-Stafford description of $\mathsf{Cent}_{\ml D(\g)}(\bbk[\g]^G)$ \cite[Theorem\,1.1]{ls}, see discussion in
Section~\ref{sect:diff_op}.
\noindent
On the other hand, the formulation given in Remark~\ref{istolkov} suggests 
the following question,  which was raised by Dixmier himself \cite[1.2]{dixm}.
\begin{qtn}
Suppose $D\in \ml D(\g)$ and $D(f)=0$ for all $f\in \bbk[\g]^G$. Is it true 
that $D\in \ml D(\g){\cdot}\boldsymbol{\varsigma}(\g)$?
\end{qtn}

\noindent
The affirmative answer  is obtained by Levasseur and Stafford \cite{ls96}. 
They proved that
\[
\mathcal K:=\{D\in \ml D(\g)\mid D(f)=0 \ \ \forall f\in \bbk[\g]^G\} 
\]
is the left ideal of $\ml D(\g)$ generated by $\boldsymbol{\varsigma}(\g)$.  
Then a similar result was 
obtained for the isotropy representation of some symmetric pairs \cite{ls99}.
To state that result, we need some preparations.
Let $\Sigma$ be the restricted root system of $(\g,\g_0)$. The following condition on
$\Sigma$
was  considered by Sekiguchi \cite{sek}:  \\[.6ex]
\hbox to \textwidth{\ $(\heartsuit)$ \hfil
     $\dim\g_\ap+\dim\g_{2\ap}\le 2$  \  for any  \ $\ap\in \Sigma$. \hfil} 
\vskip.6ex\noindent
Sekiguchi also obtained the list of corresponding symmetric pairs. 
The following is \cite[Theorem\,C]{ls99}:

\noindent
{\it Suppose that $\Sigma(\g,\g_0)$ satisfies $(\heartsuit)$.
Then 
$\mathcal K(\g_1):=\{D\in \ml D(\g_1)\mid D(f)=0 \ \ \forall f\in \bbk[\g_1]^{G_0}\}$ is the left 
ideal generated by $\boldsymbol{\varsigma}(\g_0)$.}

\noindent
Furthermore, it is proved in \cite[Section\,6]{ls99} that 
$\mathcal K(\g_1)=\ml D(\g_1)\boldsymbol{\varsigma}(\g_0)$ for 
$(\mathfrak{so}_{n+1}, \mathfrak{so}_{n})$, while the inclusion 
$\ml D(\g_1)\boldsymbol{\varsigma}(\g_0)\subset \mathcal K(\g_1)$ is strict for
$(\mathfrak{sl}_3,\mathfrak{gl}_2)$.
It is curious that  according to Sekiguchi's classification, $(\heartsuit)$ is satisfied precisely if 
$(\g,\g_0)$ is $\N$-regular except for 
$(\mathfrak{sl}_{2n+1}, \mathfrak{sl}_n\dotplus\mathfrak{sl}_{n+1}\dotplus \te_1)$.

This raises the following  questions for  representations
of connected (reductive?) groups.
There are two properties of representations $(G:V)$:

1) \ The algebra $\mathsf{Cent}_{\ml D(V)}(\bbk[V]^G)$ is generated by $\bbk[V]$ and 
$\boldsymbol{\varsigma}_V(\g)$;

2) \ The ideal $\{D\in \ml D(V)\mid D(f)=0 \ \ \forall f\in \bbk[V]^{G}\}$ is generated 
by $\boldsymbol{\varsigma}_V(\g)$.
\\
Is it true that one of them implies
another (under appropriate constraints)? At least, is there a relationship in case of
isotropy representations of symmetric pairs?

\begin{rem}  The only "bad" $\N$-regular symmetric pair
$(\mathfrak{sl}_{2n+1}, \mathfrak{sl}_n\dotplus\mathfrak{sl}_{n+1}\dotplus \te_1)$
is also distinguished by a bad behaviour of the commuting variety.
Recall that the commuting variety  is 
$\mathfrak E(\g_1)=\{ (x,y)\in \g_1\times\g_1\mid [x,y]=0\}$, and it
is irreducible for all $\N$-regular pairs but that one.
In the maximal rank case, the irreducibility is proved in \cite{jac}.
The four remaining cases (see the list in Section~\ref{sect:z2}) are considered in
\cite{py1}. This is of certain interest because there is a relationship between
the irreducibility of  $\mathfrak E(\g_1)$ and properties of the ideal
$\mathcal K(\g_1)$, see \cite[Prop.\,4.6]{ls99}.

\end{rem}

\end{document}